\documentclass[11pt]{amsart}
\usepackage{amsmath, amsfonts, amssymb, amsthm}
\usepackage{enumerate}
\usepackage{graphicx}
\usepackage{xcolor}
\usepackage[colorlinks,citecolor=blue]{hyperref}
\usepackage{verbatim}
\usepackage[headings]{fullpage}
\usepackage{calc}
\usepackage[all, cmtip]{xy}
\usepackage{thmtools}
\usepackage{thm-restate}

\usepackage{mathtools}

\newtheorem{theorem}{Theorem}[section]
\newtheorem{lemma}[theorem]{Lemma}
\newtheorem{proposition}[theorem]{Proposition}

\newtheorem{corollary}[theorem]{Corollary}
\newcounter{theoremalph}

\newtheorem{thmAlph}[theoremalph]{Theorem}
\theoremstyle{definition}
\newtheorem{definition}[theorem]{Definition}
\newtheorem{example}[theorem]{Example}
\newtheorem{question}[theorem]{Question}

\theoremstyle{remark}
\newtheorem{remark}[theorem]{Remark}

\numberwithin{equation}{section}

\DeclareMathOperator{\Out}{Out}
\DeclareMathOperator{\Aut}{Aut}

\newcommand{\rank}{\mathfrak{n}}
\newcommand{\op}{\operatorname}

\newcommand{\la}{\langle}
\newcommand{\ra}{\rangle}

\newcommand{\FreeS}[1][\rank]{\mathcal{FS}_{#1}}

\newcommand{\bFreeS}{\mathcal{FS}^*}
\newcommand{\FreeSone}{\mathcal{FS}^1}
\newcommand{\FreeSred}[1][\rank]{\mathcal{FS}_{#1}^{r}}
\newcommand{\bFreeSred}{\mathcal{FS}^{r,*}}

\newcommand{\FreeF}[1][\rank]{\mathcal{F}_{#1}}
\newcommand{\FFS}[1][\rank]{\mathcal{FF}_{#1}}

\newcommand{\ffa}{\mathcal{A}}
\newcommand{\ffb}{\mathcal{B}}

\newcommand{\free}{\mathbb{F}}
\newcommand{\CV}[1][\rank]{\mathcal{CV}_{#1}}

\newcommand{\CVbound}{\partial_s \CV}
\newcommand{\CVred}[1][\rank]{\CV[#1]^{r}}
\newcommand{\jewel}{\mathcal{J}_\rank}

\newcommand{\set}[1][ ]{\ensuremath{ \lbrace #1 \rbrace}}
\newcommand{\bsl}{\ensuremath{\setminus}}
\newcommand{\mcC}{\ensuremath{\mathcal{C}}}
\newcommand{\mcP}{\ensuremath{\mathcal{P}}}
\newcommand{\mcZ}{\ensuremath{\mathcal{Z}}}

\newcommand{\V}{\ensuremath{\mathcal{V}}}
\newcommand{\D}{\ensuremath{\mathcal{D}}}

\newcommand{\Sub}{\ensuremath{\operatorname{Sub}}}

\newcommand{\X}{\ensuremath{\operatorname{X}}}
\newcommand{\cX}{\ensuremath{\operatorname{cX}}}
\newcommand{\Forest}{\ensuremath{\operatorname{For}}}
\newcommand{\Core}{\ensuremath{\operatorname{C}}}
\newcommand{\cCore}{\ensuremath{\operatorname{cC}}}

\newcommand{\real}[1]{\ensuremath{\left\lVert #1\right\rVert}}

\title{Homotopy type of the complex of free factors of a free group}
\author{Benjamin Br\"{u}ck}
\address{\tt B.\ Br\"uck, University of Copenhagen, Department of Mathematical Sciences, \,\newline Universitetsparken 5, DK-2100 Copenhagen, Denmark  \newline
 http://www.math.uni-bielefeld.de/\~{}bbrueck/} %
\email{\tt benjamin.brueck.maths@gmail.com}

\author{Radhika Gupta }
\address{\tt R.\ Gupta, Temple University, Department of Mathematics \,\newline Wachman Hall, 1805 North Broad Street, Philadelphia, PA 19122, Pennsylvania, USA
\newline
https://sites.google.com/view/radhikag/} %
\email{\tt radhikagupta.maths@gmail.com}
\thanks{\today}

\begin{document}

\begin{abstract}
We show that the complex of free factors of a free group of rank $\rank \geq 2$ is homotopy equivalent to a wedge of spheres of dimension $\rank-2$. We also prove that for $\rank \geq 2$, the complement of (unreduced) Outer space in the free splitting complex is homotopy equivalent to the complex of free factor systems and moreover is $(\rank-2)$-connected. In addition, we show that for every non-trivial free factor system of a free group, the corresponding relative free splitting complex is contractible. 
\end{abstract}


\maketitle
\section{Introduction}\label{sec:intro}
Let $\free$ be the free group of finite rank $\rank$. A free factor of $\free$ is a subgroup $A$ such that $\free = A \ast B$ for some subgroup $B$ of $\free$. Let $[.]$ denote the conjugacy class of a subgroup of $\free$. Define $\FreeF$ to be the partially ordered set (poset) of conjugacy classes of proper, non-trivial free factors of $\free$ where $[A]\leq [B]$ if for suitable representatives, one has $A\subseteq B$. We will call the order complex (geometric realisation) of this poset the \emph{complex of free factors} or \emph{free factor complex} and denote it also by $\FreeF$. Since a maximal nested chain of conjugacy classes of free factors has length $\rank-2$ (see Section~\ref{sec:posetprelim} for the notational conventions we use),
 $\FreeF$ is $(\rank-2)$-dimensional. Note that for $\rank =2$, our definition differs from the usual one: commonly, two conjugacy classes of free factors of $\free_2$ are connected by an edge in $\FreeF[2]$ if they have representatives that span a basis.
There is a natural action of the group of outer automorphisms of $\free$, denoted $\Out(\free)$, on $\FreeF$. The geometry of this complex has been studied very well in recent years and it was used to improve the understanding of $\Out(\free)$. Most notably, Bestvina and Feighn in \cite{BF:FreeFactorComplex} showed that $\FreeF$ is Gromov-hyperbolic, in analogy to Masur--Minsky's hyperbolicity result for the curve complex of a surface \cite{MM:CurveComplex}. In this paper, we investigate the topology of $\FreeF$. Our main result is as follows: 

\begin{thmAlph}\label{thm:FreeFactorComplex}
For $\rank \geq 2$, the free factor complex $\FreeF$ is homotopy equivalent to a countable infinite wedge of spheres of dimension $\rank-2$. 
\end{thmAlph}

In \cite{HV:ComplexOfFreeFactors}, Hatcher and Vogtmann showed that  the geometric realisation of the poset of proper free factors in $\free$ is homotopy equivalent to a wedge of spheres of dimension $\rank-2$. Note that Hatcher and Vogtmann's complex is different from the free factor complex $\FreeF$ in that its vertices are proper free factors and not conjugacy classes of proper free factors. Since $\FreeF$ comes with a natural action of $\Out(\free)$ instead of $\Aut(\free)$, the focus has shifted more towards this version over the years.

\subsection*{Motivation} The motivation for describing the homotopy type of this and similar factor complexes comes from the analogy with the rational Tits building, $\Delta(n, \mathbb{Q})$, associated to $\operatorname{SL}_n(\mathbb{Z})$. The definition of these $\Out(\free)$-simplicial complexes is similar to $\Delta(n,\mathbb{Q})$. By the Solomon--Tits theorem (\cite{SolomonTits}), the rational Tits building is homotopy equivalent to a wedge of spheres of dimension $n-2$. In \cite{BorelSerre}, Borel and Serre used this to show that the dualising module of any torsion free finite index subgroup $\Gamma$ of $\operatorname{SL}_n(\mathbb{Z})$ is $\widetilde{H}_{n-2}(\Delta(n,\mathbb{Q}), \mathbb{Z})=:D$, that is, $$H^i(\Gamma,M)\cong H_{d-i}(\Gamma, M \otimes D)$$ for any $\Gamma$-module $M$, for all $i>0$, and $d$ equal to the virtual cohomological dimension of $\operatorname{SL}_n(\mathbb{Z})$. 

This relationship between $\operatorname{SL}_n(\mathbb{Z})$ and $\Delta(n,\mathbb{Q})$ has been successfully extended to the mapping class group of an orientable surface of genus $g$ and $p$ punctures $\op{MCG}(\Sigma_{g,p})$ and the associated curve complex $\mathcal{C}(\Sigma_{g,p})$. Harvey (\cite{Harvey}) defined the curve complex and showed that it can be seen as a boundary structure of Teichm\"uller space. Harer (\cite{H:VCD} for punctured surfaces) and Ivanov (\cite{Ivanov} for closed surfaces) showed that this complex is homotopy equivalent to a wedge of spheres. As a consequence, the top dimensional reduced homology of the curve complex is the dualising module for the mapping class group.

The rational Tits building and the curve complex capture the asymptotic geometry of the symmetric space for $\operatorname{SL}_n(\mathbb{Z})$ and Teichm\"{u}ller space for $\op{MCG}(\Sigma)$, respectively. In this paper, we also obtain some partial results about the asymptotic geometry of Culler--Vogtmann's Outer space $\CV$. 

\subsection*{Asymptotic geometry of Outer space}
Let $\CV$ be Culler--Vogtmann's Outer space and $L$ be its spine. We denote by $\CVred$ reduced Outer space which has spine $K$. Let $\FreeS$ be the free splitting complex. For the definitions, see Section~\ref{sec:prelimOuterspace}. We view $L, K$ and $\FreeS$ as partially ordered sets and use the same notation for both the poset and its order complex. As explained in Section~\ref{sec:prelimOuterspace}, $\FreeS$ is the simplicial completion of $\CV$ and $L$ is a subposet of $\FreeS$.  Analogously, there is a natural subposet $\FreeSred$ of $\FreeS$ which is the simplicial completion of $\CVred$ and contains $K$.  

Before we state our next theorem, we consider another poset $\FFS$ (defined in \cite{HM:RelativeFree}), whose order complex is called the \emph{complex of free factor systems} (also denoted by $\FFS$).  A \emph{free factor system} of $\free$ is a finite collection of the form $\ffa = \{[A_1], \ldots, [A_k] \}$, where $k> 0$, each $A_i$ is a proper, non-trivial free factor of $\free$, such that there exists a free factorisation $\free = A_1 \ast \cdots \ast A_k \ast F_N$. There is a partial ordering $\sqsubseteq$ on the set of free factor systems given as follows: $\ffa \sqsubseteq \ffa^{\prime}$ if for every $[A_i] \in \ffa$ there exists $[A_j^{\prime}] \in \ffa^{\prime}$ such that $A_i \subseteq A_j^{\prime}$ up to conjugation. This \emph{poset of free factor systems} is denoted by $\FFS$.  The poset $\FreeF$ of free factors is a subposet of $\FFS$. In fact, $\FreeF$ and $\FFS$ are quasi-isometric to each other by \cite[Proposition 6.3]{HM:RelativeFree}. However, Theorem~\ref{thm:FreeFactorComplex} and the following result show that they are not homotopy equivalent. 

\begin{thmAlph}
\label{thm:freefactorsystems}
 $\FreeS \setminus L$ and $\FFS$ are homotopy equivalent. Moreover, for $\rank \geq 2$, they are $(\rank-2)$-connected. 
\end{thmAlph}

In order to establish the homotopy equivalence $\FreeS \setminus L \simeq \FFS$, we are led to study relative versions of the free splitting complex: Whenever one has a free splitting $S$ of $\free$, the set of conjugacy classes of non-trivial vertex stabilisers forms a free factor system $\V(S)$. Now for a free factor system $\ffa$ in $\free$, the \emph{poset of free splittings of $\free$ relative to $\ffa$}, denoted $\FreeS(\ffa)$, is the subposet of $\FreeS$ consisting of all free splittings $S\in \FreeS$ such that $\ffa\sqsubseteq \V(S)$. Its realisation is the relative free splitting complex studied in \cite{HM:RelativeFree}. Hatcher \cite{Hatcher:FSContractible} showed that $\FreeS$ is contractible. We extend this result to the relative setting and show:

\begin{thmAlph}
\label{thm:FreeSplittingIntroduction}
For any free factor system $\ffa$ of $\free$, the relative free splitting complex $\FreeS(\ffa)$ is contractible.
\end{thmAlph}

In \cite{BSV:Bordification}, Bux, Smillie and Vogtmann introduced an equivariant deformation retract of $\CVred$ called \emph{jewel space}, denoted by $\jewel$. They showed that $\jewel$ is homeomorphic to the bordification of Outer space defined by Bestvina and Feighn in \cite{BF:Bordification} and asked what the homotopy type of its boundary $\partial \jewel$ is. We mention the following result of Vogtmann to contrast the currently known results about the connectivity of the simplicial boundaries of unreduced and reduced Outer space and also because our methods for establishing Theorem~\ref{thm:FreeFactorComplex} and Theorem~\ref{thm:freefactorsystems} give an alternate proof (albeit longer) of the $(\rank-3)$-connectivity of $\FreeSred \setminus K$. 

\begin{thmAlph}\cite{V:Private}
\label{thm:BoundaryJewelSpace}
 $\FreeSred \setminus K$ and $\partial \jewel$ are homotopy equivalent. Moreover, for $\rank \geq 3$, they are $(\rank-3)$-connected. 
\end{thmAlph}

The partial results in Theorem~\ref{thm:freefactorsystems} and \ref{thm:BoundaryJewelSpace} lead to the following question. 
\begin{question} 
\label{qes:HomotopyType?}
What are the homotopy types of $\FreeS \setminus L$, $\FFS$ and $\FreeSred \setminus K$? 
\end{question}
Unfortunately, we cannot answer this. The main difficulty we are faced with is that $\FFS$ is $(2 \rank-3)$-dimensional \cite[Proposition 6.1]{HM:RelativeFree} and our method cannot be pushed to get higher connectivity results or lower the dimension. Note that the curve complex of a closed surface of genus $g$ is $(3g-4)$-dimensional but nevertheless it is homotopy equivalent to a wedge of spheres of dimension $2g-2$. For further comments on this, see Section~\ref{sec:CV_2}.

\subsection*{Methods of proof}
Various methods have been used to determine the homotopy type of some of the complexes mentioned in this introduction: Shelling orders for $\Delta(n, \mathbb{Q})$, flow arguments for $\FreeS$, spectral sequences for Hatcher--Vogtmann's complex of free factors, Morse theory for the curve complex.
In this paper, we view all our simplicial complexes as order complexes of posets and use various Quillen type fibre lemmas (see Section~\ref{sec:posetprelim} for details) to get the desired results. In particular, the following poset version of the Vietoris--Begle theorem (see \cite[Corollary 2.4]{GV:VietorisBegle}) is the main tool we use.

\begin{restatable*}[{\cite[Proposition 1.6 \& 7.6]{Q:HomotopyProperties}}]{lemma}{QuillenFibres}
\label{lem:QuillenFibres}
Let $f:P\to Q$ be a poset map.
\begin{enumerate}
\item If for all $x\in Q$, the fibre $f^{-1}(Q_{\leq x})$ is contractible, then $f$ induces a homotopy equivalence on geometric realisations.
\item If for all $x\in Q$, the fibre $f^{-1}(Q_{\leq x})$ is $n$-connected , then $P$ is $n$-connected if and only if $Q$ is $n$-connected.
\end{enumerate}
\end{restatable*}

\subsection*{Paper Outline} In Section~\ref{sec:posetprelim}, we set the notation for posets, state the various fibre lemmas and mention some results from algebraic topology which will be used later. In Section~\ref{sec:prelimOuterspace}, we define (un-)\ reduced Culler--Vogtmann Outer space, its spine and the free splitting complex. We also explain the relationship between these spaces. Section~\ref{sec:graphposet} can be read independently of the rest of the paper. It establishes the homotopy type of posets of certain subgraphs of a fixed graph. In Section~\ref{sec:freesplittingposets}, we show contractibility of the relative free splitting complexes (Theorem~\ref{thm:FreeSplittingIntroduction}). This result is used in Section~\ref{sec:homeq} to prove the homotopy equivalence of $\FreeS \setminus L$ and $\FFS$ (the first statement of Theorem~\ref{thm:freefactorsystems}). Also in Section~\ref{sec:homeq}, we show that $\FreeF$ is homotopy equivalent to the subposet of $\FreeS$, denoted $\FreeSone$, given by free splittings with exactly one non-trivial vertex group. Finally, in Section~\ref{sec:productspace}, we prove Theorem~\ref{thm:freefactorsystems}, Theorem~\ref{thm:FreeFactorComplex} and the second statement of Theorem~\ref{thm:BoundaryJewelSpace}. We close this article in Section~\ref{sec:remarks} with some remarks concerning the complex of sphere systems and relative versions of our results and give an illustration of our considerations in the case where $\rank=2$.

\subsection*{Proof outline for Theorem~\ref{thm:FreeFactorComplex}} We now describe a brief outline for the proof of Theorem~\ref{thm:FreeFactorComplex}, which also sheds some light on the structure of the paper. See Section~\ref{sec:productspace} for the detailed proof. We first establish in Proposition~\ref{prop:HomotopyEquivalence} that $\FreeSone$ and $\FreeF$ are homotopy equivalent. Consider the pair of posets $(X, Y)$ where $X = L$ and $Y = \FreeSone$. In Section~\ref{sec:productspace}, we define a particular subposet $\mcZ$ of $X \times Y$ with projection maps $p_1 \colon \mcZ \to X$ and $p_2 \colon \mcZ \to Y$. We then show that the fibres of the map $p_2$ (in the sense of Lemma~\ref{lem:QuillenFibres}) are contractible and the fibres of $p_1$ are given by posets of subgraphs which are $(\rank-3)$-connected. Applying Lemma~\ref{lem:QuillenFibres} twice then gives us that $\FreeSone$, equivalently $\FreeF$, is $(\rank-3)$-connected. Since $\FreeF$ is $(\rank-2)$-dimensional, we obtain the desired result.

For the proof of the second statement of Theorem~\ref{thm:freefactorsystems} (resp. Theorem~\ref{thm:BoundaryJewelSpace}), we consider the pair $(X, Y) = (L, \FreeS \setminus L)$ (resp. $(X, Y) = (K, \FreeSred \setminus K)$).

\subsection*{Acknowledgements} This project has benefited greatly from the discussions the second author had with Mladen Bestvina while she was a graduate student under his supervision at the University of Utah. We would like to thank Thomas Goller, Nir Lazarovich and Derrick Wigglesworth for helpful conversations. We are also grateful to Karen Vogtmann for enlightening conversations and thank Kai-Uwe Bux for reading drafts of this paper and in particular for helping to improve the exposition. We thank the organizers of the conference `Geometry of outer space and outer automorphism groups', Warwick 2018 which gave us the opportunity to start this collaboration.
We would like to thank the anonymous referees for useful comments that helped to improve the article; in particular, these allowed to significantly shorten the proof of Theorem \ref{thm:FreeSplittingIntroduction}.

The first-named author was supported by the grant BU 1224/2-1 within the Priority Programme 2026 ``Geometry at infinity'' of the German Science Foundation (DFG).
The second-named author was supported by the Israel Science Foundation (grant 1026/15).

\section{Preliminaries on poset topology}\label{sec:posetprelim}
\label{sec:prelim_poset_topology}
Let $P=(P,\leq)$ be a poset (partially ordered set). If $x\in P$, the sets $P_{\leq x}$ and $P_{\geq x}$ are defined by
\begin{align*}
P_{\leq x}\coloneqq\set[y\in P\mid y\leq x],&& P_{\geq x} \coloneqq\set[y\in P\mid y\geq x].
\end{align*}
A \emph{chain of length $l$} in $P$ is a totally ordered subset $x_0<x_1<\ldots <x_l$.
For each poset $P=(P,\leq)$, one has an associated simplicial complex $\Delta(P)$ called the \emph{order complex} of $P$. Its vertices are the elements of $P$ and higher dimensional simplices are given by the chains of $P$. When we speak about the \emph{realisation of the poset $P$}, we mean the geometric realisations of its order complex and denote this space by $\real{ P } \coloneqq \real{ \Delta(P) }$.
With an abuse of notation, we will attribute topological properties (e.g. homotopy groups and connectivity properties) to a poset when we mean that its realisation has these properties.

A map $f:P\to Q$ between two posets is called a \emph{poset map} if $x\leq y$ implies $f(x)\leq f(y)$. Such a poset map induces a simplicial map from $\Delta(P)$ to $\Delta(Q)$ and hence a continuous map on the realisations of the posets. It will be denoted by $\real{ f }$ or just by $f$ if what is meant is clear from the context.

The \emph{direct product} $P \times Q$ of two posets $P$ and $Q$ is the poset whose underlying set is the Cartesian product $\{(p,q) \mid p \in P,\, q \in Q\}$ and whose order relation is given by
$$(p,q) \leq_{P \times Q} (p',q') \text{ if } p \leq_P p' \text{ and } q \leq_Q q'. $$

\subsection{Fibre theorems}
An important tool to study the topology of posets is given by so called fibre lemmas comparing the connectivity properties of posets $P$ and $Q$ by analysing the fibres of a poset map between them. The first such fibre theorem appeared in {\cite[Theorem A]{Q:HigherAlgebraic}} and is know as Quillen's fibre lemma. For this text, we need the following version of it:

\QuillenFibres

For a poset $P=(P,\leq)$, let $P^{op}=(P,\leq_{op})$ be the poset defined by $x\leq_{op} y :\Leftrightarrow y\leq x$.
Using the fact that one has a natural identification $\Delta(P)\cong \Delta(P^{op})$, one can draw the same conclusion as in the last lemma if one shows that $f^{-1}(Q_{\geq x})$ is $k$-connected for all $x\in Q$.

Another result that we will frequently use is:

\begin{lemma}[\cite{Q:HomotopyProperties}]
\label{lem:HomotopicPosetMaps}
If $f,g:P\to Q$ are poset maps such that $f(x)\leq g(x)$ for all $x\in P$, then they induce homotopic maps $\real{ f},\, \real{ g}$ on geometric realisations. In particular, if $f:P\to P$ is \emph{monotone}, i.e. $f(x)\leq x$ for all $x\in P$ or $f(x)\geq x$ for all $x\in P$, then $\real{ f}$ is homotopic to the identity.
\end{lemma}

Usually, the connectivity results one can obtain using fibre lemmas is bounded above by the degree of connectivity of the fibre. The following lemma gives a sufficient condition for obtaining a slightly better degree of connectivity. 
We will make use of it in Section~\ref{sec:freefactorsystems}. 

\begin{lemma}\label{lem:UpConnectivitiyByOne} Let $f \colon P \to Q$ be a poset map where $Q$ is $(k+1)$-connected. Assume that for all $q \in Q$, the fibre $f^{-1}(Q_{\leq q})$ is $k$-connected and the map $g_{\ast} \colon \pi_{k+1}(f^{-1}(Q_{\leq q})) \to \pi_{k+1}(P)$ induced by the inclusion $g \colon f^{-1}(Q_{\leq q}) \hookrightarrow P$ is trivial. Then $P$ is $(k+1)$-connected. \end{lemma}
\begin{proof}
Applying Lemma~\ref{lem:QuillenFibres}, one gets that $P$ is $k$-connected.

We now show that $\pi_{k+1}(P)$ also vanishes, which implies that $P$ is in fact $(k+1)$-connected. Consider a map $i \colon S^{k+1} \to \real{P}$ from the $(k+1)$-sphere to $P$. Using simplicial approximation (see e.g. \cite[Chapter 3.4]{Spanier:AlgebraicTopology}) we can (after possibly precomposing with a homotopy) assume that $i$ is simplicial with respect to a simplicial structure $\tau$ on $S^{k+1}$. We wish to show that $i$ extends to a map $\hat{i} \colon B^{k+2} \to \real{P}$, where $B^{k+2}$ is the $(k+2)$-ball and $\hat{i}|_{\partial B^{k+2}} = i$. 

Consider the simplicial map $h:= f \circ i \colon S^{k+1} \to \real{Q}$. Since $Q$ is $(k+1)$-connected, it extends to a map $\hat{h} \colon B^{k+2} \to \real{Q}$ such that $\hat{h}|_{\partial B^{k+2}} = h$. 
Simplicial approximation applied to the pair $(B^{k+2}, S^{k+1})$ allows us to assume that $\hat{h}$ is simplicial with respect to a simplicial structure $\tau'$ on $B^{k+2}$ such that $\tau'$ agrees with $\tau$ on $\partial B^{k+2}= S^{k+1}$. For this, we might need to do barycentric subdivision and replace $i$ by a homotopic map again.
We now show that $\hat{h}$ lifts to a map $\tilde{h} \colon B^{k+2} \to \real{P}$ such that $\tilde{h}|_{\partial B^{k+2}} = i$ by defining $\hat{h}$ inductively on the simplices of $\tau'$. 

We do induction on the skeleta of $\tau'$. To start, let $v$ be a vertex of $\tau'$. If $v \in \tau$, then $\tilde{h}(v) \coloneqq i(v)$; otherwise set $\tilde{h}(v)$ to be any vertex in $f^{-1}(\hat{h}(v))$. 
Now assume that for $m \leq k+1$, the map $\tilde{h}$ has been defined on the $(m-1)$-skeleton such that $\tilde{h}$ restricts to $i$ on $\tau$ and for every $(m-1)$-simplex $\sigma_{m-1}$ of $\tau'$, we have
\begin{equation*}
\tilde{h}(\sigma_{m-1}) \subseteq \real{f^{-1}(Q_{\leq \max{\sigma_{m-1}}})},
\end{equation*}
where $\max{\sigma_{m-1}}$ is the largest vertex in $\hat{h}(\sigma_{m-1})$. Let $\sigma_m$ be an $m$-simplex of $\tau'$. Clearly, we have $\max \sigma_m \geq \max\sigma_{m-1}$ for every $(m-1)$-face $\sigma_{m-1}$ of $\sigma_m$. Hence, $\tilde{h}(\partial \sigma_m) \subseteq f^{-1}(Q_{\leq \max \sigma_m})$ and by assumption, $f^{-1}(Q_{\leq \max \sigma_m})$ is $k$-connected. Thus $\tilde{h}$ extends to $\sigma_m$ such that $\tilde{h}(\sigma_m) \subseteq f^{-1}(Q_{\leq \max \sigma_m})$. Finally, for a $(k+2)$-simplex $\sigma$, we have $\tilde{h}(\partial \sigma) \subseteq f^{-1}(Q_{\leq \max\sigma})$. Since the image $g_{\ast}(\pi_{k+1}(f^{-1}(Q_{\leq \max\sigma})))$ in $\pi_{k+1}(P)$ is trivial, the map $\tilde{h}$ extends to $\sigma$. 
Thus we have shown that $P$ is $(k+1)$-connected. 
\end{proof}

\subsection{The nerve of a covering}
The \emph{nerve} of a family of sets $(X_i)_{i\in I}$ is the simplicial complex $\mathcal{N}(X_i)_{i\in I}$ that has vertex set $I$ and where a finite subset $\sigma\subseteq I$ forms a simplex if and only if $\bigcap_{i\in\sigma} X_i\not=\emptyset$.
The \emph{Nerve Theorem} is another standard tool which exists in various versions. For simplicial complexes, it can be stated as follows:
\begin{lemma}[{\cite[Theorem 10.6]{Bjorner:TopologicalMethods}}]
\label{lem:NerveTheorem}
Let $X$ be a simplicial complex and $(X_i)_{i\in I}$ a family of subcomplexes such that $X=\bigcup_{i\in I} X_i$. Suppose that every non-empty finite intersection $X_{i_1}\cap \ldots \cap X_{i_k}$ is contractible. Then $X$ is homotopy equivalent to the nerve $\mathcal{N}((X_i)_{i\in I})$.
\end{lemma}

\subsection{Alexander duality and the Whitehead theorem}
Alexander duality allows one to compute homology groups of compact subspaces of spheres by looking at the homology of their complement. We will need the following poset version of it which is due to Stanley.
\begin{lemma}[\cite{S:SomeAspects}, {\cite[Theorem 5.1.1]{W:PosetTopology}}]
\label{lem:AlexanderDuality}
Let $P$ be a poset such that $\real{P}$ is homeomorphic to an $n$-sphere and let $Q\subset P$ be a subposet. Then for all $i$, one has
\begin{equation*}
\tilde{H}_i(\real{Q};\mathbb{Z})\cong \tilde{H}^{n-i-1}(\real{P\bsl Q};\mathbb{Z}).
\end{equation*}
\end{lemma}

In order to deduce information about the homotopy type of a space from its homology groups, we need a corollary of the theorems of Whitehead and Hurewicz.

\begin{theorem}[Hurewicz theorem {\cite[Theorem 4.32]{Hatcher:Book}}] If a space $X$ is $(n-1)$-connected, $n \geq 2$, then $\tilde{H}_i(X) = 0$ for all $0<i<n$ and $\pi_n(X)$ is isomorphic to $H_n(X)$.   \end{theorem}

\begin{theorem}[Whitehead theorem, {\cite[Corollary 4.33]{Hatcher:Book}}]
\label{thm:Whitehead}
A map $f\colon X \to Y$ between simply-connected CW-complexes is a homotopy equivalence if $f_{\ast} \colon H_k(X) \to H_k(Y)$ is an isomorphism for each $k$. \end{theorem}

\begin{corollary}
\label{cor:Whitehead}
Let $X$ be a simply-connected CW-complex such that
\begin{equation*}
\tilde{H}_i(X)=
\begin{cases}
\mathbb{Z}^\lambda	&,\, i=n,\\
0					&,\, \text{otherwise.}
\end{cases}
\end{equation*}
Then $X$ is homotopy equivalent to a wedge of $\lambda$ spheres of dimension $n$.
\end{corollary}
\begin{proof}
By the Hurewicz theorem, $X$ is in fact $(n-1)$-connected and $\pi_n(X)\cong \tilde{H}_n(X)=\mathbb{Z}^\lambda$. Now take a disjoint union $\bigsqcup_{\mu\leq \lambda} S_\mu$ of $n$-spheres.
For each $\mu\leq \lambda$, choose a generator $S_\mu\to X$ of the $\mu$-th summand of $\pi_n(X)$.
This gives rise to a map $f:Y\to X$ where $Y$ is the space obtained by wedging together the $S_\mu$ along their base points. This induces an isomorphism $f_*$ on all homology groups, so the claim follows from the Whitehead theorem.
\end{proof}

\begin{remark}
A CW complex is \emph{$n$-spherical} if it is homotopy equivalent to a wedge of $n$-spheres. By the preceding theorems, an $n$-dimensional complex $X$ is $n$-spherical if and only if $\pi_i(X)$ is trivial for all $i<n$.
\end{remark}

\section{Outer space and its relatives}\label{sec:prelimOuterspace}

Throughout this section, let $\free$ be a free group of finite rank $\rank \geq 2$.

\subsection{Outer Space, its spine and the free splitting complex}

Identify $\free$ with $\pi_1(\mathcal{R}, \ast)$ where $\mathcal{R}$ is a rose with $\rank$ petals. A \emph{marked graph} $G$ is a graph of rank $\rank$ equipped with a homotopy equivalence $m : \mathcal{R}\to G$ called a \emph{marking}. The marking determines an identification of $\free$ with $\pi_1(G,m(\ast))$.

(Unreduced) Culler--Vogtmann \emph{Outer space} $\CV$, defined in \cite{CV:OuterSpace}, is the space of equivalence classes of marked metric graphs $G$ of volume one such that every vertex of $G$ has valence at least three. Outer space can be decomposed into a disjoint union of open simplices, where the missing faces are thought of as ``sitting at infinity". There is a natural \emph{simplicial completion} obtained by adding the missing faces at infinity. The subspace of this completion consisting of all the open faces sitting at infinity is called the \emph{simplicial boundary} $\CVbound$ of Outer space. 

A \emph{free splitting} $S$ of $\free$ is a non-trivial, minimal, simplicial $\free$-tree with trivial edge stabilisers.
The \emph{vertex group system} of a free splitting $S$ is the (finite) set of conjugacy classes of its vertex stabilisers. 
 Two free splittings $S$ and $S'$ are equivalent if they are equivariantly isomorphic. We say that $S'$ collapses to $S$ if there is a \emph{collapse map} $S'\to S$ which collapses an $\free$-invariant set of edges. The \emph{poset of free splittings} $\FreeS$ is given by the set of all equivalence classes of free splittings of $\free$ where $S\leq S'$ if $S'$ collapses to $S$. The \emph{free splitting complex} is the order complex $\Delta(\FreeS)$ of the poset of free splittings.
Outer space naturally embeds as a subspace of $\real{\FreeS}$. In fact, the free splitting complex is naturally identified with the barycentric subdivision of the simplicial completion of $\CV$.
Each free splitting $S$ can equivalently be seen as a graph of groups decomposition of $\free$ with trivial edge groups by taking the quotient $S/\free$. We will often adopt this point of view in later sections without further notice.

The \emph{spine} $L$ of $\CV$ is given by the subposet of $\FreeS$ consisting of all free splittings that have trivial vertex stabilisers. We can interpret $\real{L}$ as a subspace of $\CV$. It consists of all marked metric graphs $G$ satisfying the following property:
The subgraph spanned by the set of all edges of $G$ not having maximal length forms a forest.
In \cite{CV:OuterSpace}, Culler and Vogtmann showed that $L$ is a contractible deformation retract of $\CV$.
By the definitions above, we have a homeomorphism
\begin{equation*}
\CVbound \cong \real{\FreeS\bsl L}.
\end{equation*}

An edge $e$ of a graph $G$ is called a \emph{separating edge} if removing it from $G$ results in a disconnected graph. The subspace of $\CV$ consisting of all marked graphs that do not contain separating edges is called \emph{reduced Outer space}, denoted $\CVred$. It is an equivariant deformation retract of $\CV$. Similarly to the unreduced cased, there is a poset $K$ such that $\CVred$ retracts to $\real{K}$. It is the subposet of $L$ consisting of all marked graphs having no separating edges and is called the \emph{spine (of reduced Outer space)}.

The barycentric subdivision of the simplicial closure of reduced Outer space is given by the order complex (see Section \ref{sec:prelim_poset_topology}) of the poset $\FreeSred$ consisting of all those free splittings $S\in \FreeS$ such that the quotient $S/\free$ does not have any separating edges.
Just as in the unreduced case, we have
\begin{equation*}
\partial_s\CVred \cong \real{\FreeSred\bsl K}.
\end{equation*}

\subsection{Relative Outer space and its spine} In \cite{GL:OuterSpaceFreeProduct}, Guirardel and Levitt define relative Outer space for a countable group that splits as a free product $G = G_1 \ast \ldots \ast G_k \ast F_N$ where $N+k \geq 2$. They also prove contractibility of relative Outer space. We will later on consider the case where $G = \free$ splits as $\free = A_1 \ast \ldots \ast A_k \ast F_N$ for $k> 0$. Let $\ffa = \{[A_1], \ldots, [A_k]\}$ be the associated free factor system of $\free$.

Subgroups of $\free$ that are conjugate into a free factor in $\ffa$ are called \emph{peripheral} subgroups. An $(\free, \ffa)$-tree is an $\mathbb{R}$-tree with an isometric action of $\free$, in which every peripheral subgroup fixes a unique point. Two $(\free, \ffa)$-trees are equivalent if there exists an $\free$-equivariant isometry between them.  A \emph{Grushko} $(\free,\ffa)$-graph is the quotient by $\free$ of a minimal, simplicial metric $(\free, \ffa)$-tree, whose set of point stabilisers is the free factor system $\ffa$ and edge stabilisers are trivial.
 \emph{Relative Outer space} is the space of homothety classes of equivalence classes of Grushko $(\free, \ffa)$-graphs.
The \emph{spine of relative Outer space}, denoted by $L(\free, \ffa)$, is the subposet of $\FreeS$ consisting of all free splittings whose system of vertex stabilisers is given by $\ffa$. Its realisation can be seen as a subspace of relative Outer space.
Since relative Outer space deformation retracts onto its spine, $L(\free, \ffa)$ is contractible.

\section{Posets of graphs}\label{sec:graphposet}
In this section, we study (finite) posets of subgraphs of a given graph $G$.
For the combinatorial arguments we use, let us set up the following notation:

In what follows, all graphs are assumed to be finite. They are allowed to have loops and multiple edges. For a graph $G$, we denote the set of its vertices by $V(G)$ and the set of its edges by $E(G)$.
If $e\in E(G)$ is an edge, then $G-e$ is defined to be the graph obtained from $G$ by removing $e$ and $G/e$ is obtained by collapsing $e$ and identifying its two endpoints to a new vertex $v_e$.
A graph is called a \emph{tree} if it is contractible. It is called a \emph{forest} if it is a disjoint union of trees.

Throughout this section, we will only care about \emph{edge-induced subgraphs}, i.e. when we talk about a ``subgraph $H$ of $G$'', we will always assume that $H$ is possibly disconnected but does not contain any isolated vertices. Hence, we can interpret any subgraph of $G$ as a subset of $E(G)$.

\begin{definition}
A \emph{core subgraph} $H$ of a graph $G$ is a proper subgraph such that the fundamental group of each connected component of $H$ is non-trivial and no vertex of $H$ has valence one in $H$.
Every graph $G$ contains a unique maximal core subgraph that we will refer to as the \emph{core of $G$}, denoted by $\mathring{G}$.
\end{definition}
Note that, in contrast to the convention introduced in \cite{BF:Bordification}, our core subgraphs are allowed to have separating edges.

\subsection{The poset of all core subgraphs}

\begin{definition}
Let $G$ be a graph. We define the following posets of subgraphs of $G$; all of them are ordered by inclusion:
\begin{enumerate}
\item $\Sub(G)$ is the poset of all proper subgraphs of $G$ that are non-empty. Equivalently, $\Sub(G)$ can be seen as the poset of all proper, non-empty subsets of $E(G)$.
\item $\Forest(G)$ denotes the poset of all proper, non-empty subgraphs of $G$ that are forests.
\item $\X(G)$ is defined to be the poset of proper subgraphs of $G$ that are non-empty and where at least one connected component has non-trivial fundamental group.
\item $\Core(G)$ is the poset of all proper core subgraphs of $G$.
\end{enumerate}
\end{definition}

Clearly one has:
\begin{equation*}
\Core(G)\subseteq\X(G)\subseteq\Sub(G)
\end{equation*}
and
\begin{equation*}
\X(G)=\Sub(G)\bsl\Forest(G).
\end{equation*}
Examples of the realisation of $\X(G)$ can be found in the Appendix, see Figure~\ref{fig:X(G) 4 edges}.
 
The proof of the following lemma is fairly standard and we will use the argument several times throughout this article. For the sake of completeness, here we will spell it out once.
\begin{lemma}
\label{lem:SubgraphstoCores}
$\X(G)$ deformation retracts to $\Core(G)$.
\end{lemma}
\begin{proof}
Every subgraph $H\in \X(G)$ contains a unique maximal core subgraph $\mathring{H}$ and if $H_1\subseteq H_2$, one has $\mathring{H_1}\subseteq \mathring{H_2}$. Hence, sending each $H$ to this core subgraph $\mathring{H}$ defines a poset map $f:\nolinebreak\X(G)\to \Core(G)$ restricting to the identity on $\Core(G)$. Let $\iota$ denote the inclusion $\Core(G)\hookrightarrow \X(G)$. Then the composition $\iota\circ f:\X(G)\to\nolinebreak\X(G)$ clearly satisfies $\iota\circ f(H)\leq H$ for all $H\in \X(G)$ which by Lemma~\ref{lem:HomotopicPosetMaps} implies that it is homotopic to the identity. As $\iota\circ f|_{\Core(G)}\equiv\operatorname{id}$, this finishes the proof.
\end{proof}

\begin{proposition}
\label{prop:X(G) homotopy}
Let $G$ be a finite connected graph whose fundamental group has rank $\rank\geq 2$ and assume that every vertex of $G$ has valence at least $3$.
Then $\X(G)$ is contractible if and only if $G$ has a separating edge. If $G$ does not have a separating edge, then $\X(G)$ is homotopy equivalent to a wedge of spheres of dimension $\rank-2$.
\end{proposition}
\begin{proof}
Note that $\Sub(G)$ can be seen as the poset of all proper faces of a simplex with vertex set $E(G)$. Hence, its realisation $\real{\Sub(G)}$ is homeomorphic to a sphere of dimension $|E(G)|-2$.

By \cite[Proposition 2.2]{V:LocalStructures}, the poset $\Forest(G)$ is contractible if and only if $G$ has a separating edge and is homotopy equivalent to a wedge of $(|V(G)|-2)$-spheres if it does not contain a separating edge.
We want to use Alexander duality as stated in Lemma~\ref{lem:AlexanderDuality} to describe the homology groups of $\X(G)=\Sub(G)\bsl\Forest(G)$.

If $G$ has a separating edge, it immediately follows from Alexander duality that all reduced homology groups of $\X(G)$ vanish.
If on the other hand $G$ does not have a separating edge, then the only non-trivial homology group of $\X(G)$ appears in dimension
\begin{equation*}
(|E(G)|-2)-1-(|V(G)|-2) = \rank-2
\end{equation*}
where it is given by a direct sum of $\mathbb{Z}$'s.

We next want to show that for $\rank \geq 4$, the realisation of $\X(G)$ is simply-connected in order to apply the Whitehead theorem.

Denote by $\Sub(G)^{(k)}$ the subposet of $\Sub(G)$ given by those subgraphs having precisely $(|E(G)|-k)$ edges. As $\rank\geq 4$, removing at most three edges from $G$ results in a graph with non-trivial fundamental group. Hence, we have $\Sub(G)^{(k)}\subset \X(G)$ for $k=1,2,3$. The realisation of
\begin{equation*}
\Sub(G)^{(\leq 3)}\coloneqq\Sub(G)^{(1)}\cup\Sub(G)^{(2)}\cup\Sub(G)^{(3)}
\end{equation*}
forms a subspace of $\real{\X(G)}$ that is homeomorphic to the $2$-skeleton of an $(|E(G)|-2)$-simplex. In particular, it is simply-connected.

\begin{figure}
\begin{center}
\includegraphics[scale=1]{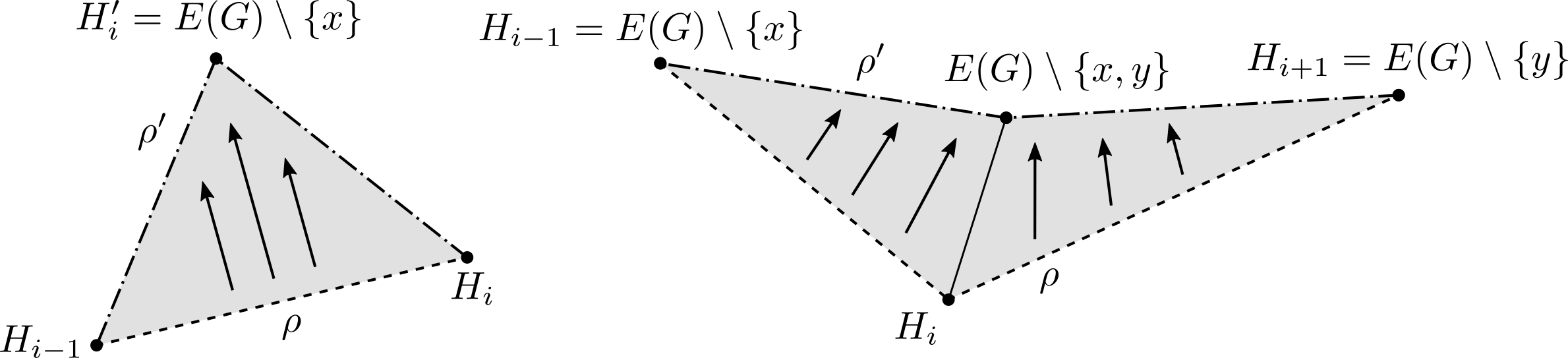}
\end{center}
\caption{Simple connectedness of $\X(G)$.}
\label{fig:Xsimply-connected}
\end{figure}

Now let $\rho$ be a closed edge path in $\real{\X(G)}$ given by the sequence of vertices $(H=H_1,H_2,...,H_k=H)$. We want to show that it can be homotoped to a path in $\real{\Sub(G)^{(1)}\cup \Sub(G)^{(2)}}$.
Whenever we have an edge $(H_{i-1}\subset H_{i})$ such that $H_{i}$ has at least two edges less than $G$, there is a subgraph $H'_{i} \in \Sub(G)^{(1)}$ containing $H_i$. As the chain $(H_{i-1}\subset H_{i}\subset H'_{i})$ forms a simplex in $\X(G)$, we can replace the segment $(H_{i-1}, H_{i})$ by $(H_{i-1},H'_i, H_{i})$ and hence assume that every second vertex crossed by $\rho$ lies in $\Sub(G)^{(1)}$ (see the left hand side of Figure~\ref{fig:Xsimply-connected}).
Next take a segment $(H_{i-1}\supset H_i \subset H_{i+1})$ where $H_{i-1}=E(G)\bsl\set[x]$ and $H_{i+1}=E(G)\bsl\set[y]$ lie in $\Sub(G)^{(1)}$. In this situation, the two chains $(H_i\subseteq E(G)\bsl\set[x,y]\subset H_{i-1})$ and $(H_i\subseteq E(G)\bsl\set[x,y]\subset H_{i+1})$ form simplices contained in $\X(G)$. It follows that we can perform a homotopy in order to replace $(H_{i-1}\supset H_i \subset H_{i+1})$ by $(H_{i-1}\supset E(G)\bsl\set[x,y] \subset H_{i+1})$.

This argument shows that every closed path can be homotoped to a path that lies in $\real{\Sub(G)^{(\leq 3)}}$. As this is a simply-connected subset of $\real{\X(G)}$, it follows that $\X(G)$ itself is simply-connected for $\rank\geq 4$. Applying Corollary~\ref{cor:Whitehead} yields the result.

The only cases that remain are those where $\rank=2$ or $3$. However, as we assumed that every vertex of $G$ has valence at least $3$, there are only finitely many such graphs. Using Lemma~\ref{lem:SubgraphstoCores}, it is not hard to verify the claim using a case-by-case analysis. For completeness, the proof for $\rank=3$ can be found in the Appendix~\ref{proof for rank=3}.
\end{proof}

\begin{remark}
Assuming that each vertex of $G$ has valence at least $3$ does not impose any restrictions for the considerations in this article as every graph in Outer space satisfies this condition.
However, note that we only used this assumption in the case where $\rank=2$ or $3$ and there it only shortened the argument and could easily be dropped.
\end{remark}

\subsection{The poset of connected core subgraphs}\label{sec:connectedcoresubgraphs}
\begin{definition}
For a graph $G$, we define $\cX(G)$ to be the poset of all proper connected subgraphs of $G$ that are not trees
\begin{equation*}
\cX(G)=\set[H \text{ subgraph of }G\mid G\not=H \text{ connected, } \pi_1(H)\not=\set[1]]
\end{equation*}
ordered by inclusion.
Let $\cCore(G)$ by the poset of all proper connected core subgraphs of $G$.
\end{definition}

Later on, we will be interested in the homotopy type of $\cCore(G)$ as it occurs as the fibre of a map we will use to study higher connectivity of $\FreeF$.
However, it is easier to describe the structure of $\cX(G)$, so we set up the following Lemma.

\begin{lemma}
\label{lem:cSubgraphstocCores}
$\cX(G)$ deformation retracts to $\cCore(G)$.
\end{lemma}
\begin{proof}
As we allow our core subgraphs to have separating edges, the unique maximal core subgraph $\mathring{H}$ contained in a \emph{connected} non-tree subgraph $H$ is connected as well. Hence, sending $H$ to $\mathring{H}$ defines a monotone poset map $\cX(G)\to \cCore(G)$. Using Lemma~\ref{lem:HomotopicPosetMaps} as in the proof of Lemma~\ref{lem:SubgraphstoCores}, the claim follows.
\end{proof}

\begin{definition}[Valence-2-homotopy]
\label{definition valence-2-homotopy}
Let $G$ be a finite connected graph and $v\in V(G)$ be a vertex of valence two with adjacent edges $e_1\not=e_2$. We define $G^v$ to be the graph obtained from $G$ by replacing the segment $e_1ve_2$ by a new edge $e_v$; i.e. 
\begin{align*}
V(G^v)=V(G)\bsl\set[v],&& E(G^v)=\set[e_v]\cup E(G)\bsl\set[e_1,e_2]
\end{align*}
and $e_v$ connects the endpoints of $e_1$ and $e_2$ that are not equal to $v$.
\end{definition}

The graphs we want to study have no vertices of valence smaller than $3$. In order to preserve this property throughout the induction procedure used in the proof of Proposition~\ref{prop:cX(G) homotopy}, we need the following: 

\begin{lemma}
\label{lemma valence-2}
Let $G$ be as in Definition~\ref{definition valence-2-homotopy}. Then $\real{\cX(G)}\simeq \real{\cX(G^v)}$.
\end{lemma}
\begin{proof}
Define poset maps $\phi:\cX(G)\to \cX(G^v)$ and $\psi:\cX(G^v)\to \cX(G)$ as follows:
\begin{equation*}
\phi(H)\coloneqq
\begin{cases}
H\bsl\set[e_1]&,\,e_1\in H,\, e_2\not\in H,\\
H\bsl\set[e_2]&,\,e_1\not\in H,\, e_2\in H,\\
\set[e_v]\cup H\bsl\set[e_1,e_2]&,\,e_1\in H\text{ and } e_2\in H,\\
H& \text{, else}.
\end{cases}
\end{equation*}
\begin{equation*}
\psi(K)\coloneqq
\begin{cases}
\set[e_1,e_2]\cup K\bsl\set[e_v]&,\,e_v\in K,\\
K& \text{, else}.
\end{cases}
\end{equation*}
By definition, one has $\psi\circ\phi(H)\subseteq H$ and $\phi\circ\psi(K)=K$, hence by Lemma~\ref{lem:HomotopicPosetMaps}, the maps $\phi$ and $\psi$ induce inverse homotopy equivalences on geometric realisations.
\end{proof}

\begin{proposition}
\label{prop:cX(G) homotopy}
Let $G$ be a finite connected graph whose fundamental group has rank $\rank\geq 2$ and assume that every vertex of $G$ has valence at least $3$. Then $\real{\cX(G)}$ is homotopy equivalent to a wedge of $(\rank-2)$-spheres.
\end{proposition}
\begin{proof}
We do induction on $\rank$. If $\rank=2$, there are exactly three graphs having only vertices of valence at least $3$. It is easy to check that for all of them, the poset of proper connected core subgraphs is a disjoint union of two or three points, i.e. a wedge of $0$-spheres.
Now assume $G$ is a graph whose fundamental group has rank $\rank>2$. If every edge of $G$ is a loop, $G$ is a rose with $\rank$ petals and every proper non-empty subset of $E(G)$ is an element of $\cX(G)$. Hence, the order complex of $\cX(G)$ is given by the set of all proper faces of a simplex of dimension $\rank-1$ whose vertices are in one-to-one correspondence with the edges of $G$.

Now assume that $G$ has an edge $e$ that is not a loop. Whenever $H\in \cX(G)$, the set $H\bsl\set[e]$ can be seen as a connected non-tree subgraph of $G/e$. If $H$ is not equal to $G-e$, then $H\bsl\set[e]$ is a proper subgraph of $G/e$. Consequently, we get a poset map
\begin{align*}
\phi: \cX(G)\bsl \set[G-e]&\to \cX(G/e)\\
H&\mapsto H\bsl \set[e].
\end{align*}

On the other hand, if we take a subgraph $K\in \cX(G/e)$ that contains the vertex $v_e$ to which $e$ was collapsed, it is easy to see that $K\cup\set[e]$ is an element of $\cX(G)\bsl \set[G-e]$.
This way, we can define a poset map
\begin{align*}
\label{map undoing a collapse}
\psi: \cX(G/e)&\to \cX(G)\bsl \set[G-e]\\
K&\mapsto
\begin{cases}
K\cup\set[e]&,\,v_e\in V(K),\\
K&\text{, else.}
\end{cases}
\end{align*}
One has $\psi\circ\phi(H)\supseteq H$ and $\phi\circ\psi(K)=K$, so using Lemma~\ref{lem:HomotopicPosetMaps}, the realisations of these two posets are homotopy equivalent.

When $e$ is a separating edge, the graph $G-e$ is not connected so in particular not an element of $\cX(G)$. It follows that $\real{\cX(G)}$ is homotopy equivalent to $\real{\cX(G/e)}$. As $G/e$ has one edge less than $G$ and every vertex in $G/e$ has valence at least $3$, we can apply induction.

If on the other hand $e$ is not a separating edge, $G-e$ is a connected graph having the same number of vertices as $G$ and one edge less. This implies that $rk(\pi_1(G-e))=\rank-1$. After possibly applying Lemma~\ref{lemma valence-2}, we may assume that each vertex in $G-e$ has valence at least $3$.

$\real{\cX(G)}$ is obtained from $\real{\cX(G)\bsl \set[G-e]}$ by attaching the star of $G-e$ along its link. The link of $G-e$ in $\real{\cX(G)}$ is naturally isomorphic to $\real{\cX(G-e)}$ which is by induction homotopy equivalent to a wedge of $(\rank-3)$-spheres. The star of a vertex is always contractible and gluing a contractible set to an $(\rank-2)$-spherical complex along an $(\rank-3)$-spherical subcomplex results in an $(\rank-2)$-spherical complex, so the claim follows (see e.g. \cite[Lemma 6.2]{BSV:Bordification}).
\end{proof}

\section{Contractibility of relative free splitting complexes}\label{sec:freesplittingposets}
\label{sec:posets of free splittings}
Throughout this section, let $\rank\geq 2$.
For a free splitting $S$, let $\V(S)$ denote its vertex group system.
Given a free factor system $\ffa$ in $\free$, the \emph{poset of free splittings of $\free$ relative to $\ffa$}, denoted $\FreeS(\ffa)$, is the subposet of $\FreeS$ consisting of all free splittings $S\in \FreeS$ such that $\ffa\sqsubseteq \V(S)$. Its realisation is the relative free splitting complex studied in \cite{HM:RelativeFree}, where the authors showed that it is non-empty, connected and hyperbolic. For a proper free factor $A$ in $\free$, let $\FreeSone_\rank([A])$ be the subposet of $\FreeS([A])$ consisting of all free splittings having exactly one non-trivial vertex group $[B]\geq [A]$. 
The aim of this section is to show that both $\FreeS(\ffa)$ and $\FreeSone([A])$ are contractible (Theorem~\ref{thm:contractibilty of relative free splitting} and Theorem~\ref{thm:contractibilty of FreeSone}).

The geometric realisation of $\FreeS(\ffa)$, denoted $\real{\FreeS(\ffa)}$, can also be described as the simplicial closure of relative outer space for $\ffa$, as defined in \cite{GL:OuterSpaceFreeProduct}. We view the relative outer space of $\ffa$ as a projectivised deformation space $\mcP\D$ and use results of \cite{GL:OuterSpaceFreeProduct} and \cite{GL:DeformationSpace} to obtain contractibility of $\FreeS(\ffa)$. We also use this technique to prove contractibility of another family of subcomplexes defined in Section~\ref{sec:FreeSone contractibility} as a step to prove Theorem~\ref{thm:contractibilty of FreeSone}. 
We now briefly recall the set up and outline for proving contractibility of deformation spaces \`{a} la Skora. We refer the reader to \cite{GL:OuterSpaceFreeProduct}, \cite{GL:DeformationSpace}, \cite{Clay:DeformationSpace} for details.  

\subsection{Contractibility \`{a} la Skora}\label{sec:Skora}

Let $\mathcal{T}$ be the set of all non-trivial minimal metric simplicial $\free$-trees. It can be given two different types of topologies, the Gromov--Hausdorff topology and the weak topology. Let $\D$ be unprojectivised relative outer space for $\ffa$. Then $\D$ is a subset of $\mathcal{T}$ and can also be given the two topologies. We denote by $\bar{\D}$ the closure of $\D$ in the space of non-trivial minimal $\mathbb{R}$-trees with action of $\free$, equipped with the Gromov--Hausdorff topology. 
In what follows, we will change views between these spaces and their projectivised versions. This is justified by the fact that in both the Gromov--Hausdorff and the weak topology, $\mathcal{T}$ is homeomorphic to $\mcP \mathcal{T}\times \mathbb{R}$, which allows to transfer results from one setting to the other (see the comments in \cite[p. 152]{GL:DeformationSpace}).

In \cite[Theorem~6.1]{GL:DeformationSpace}, it is shown that 
$\bar{\D}$ with the Gromov--Hausdorff topology is contractible. The key technique for proving contractibility is Skora's idea of deforming morphisms between metric trees: For $T, T' \in \mathcal{T}$, a map $f:T'\to T$ is called a \emph{morphism} if it is $\free$-equivariant and every edge of $T'$ can be written as a finite union of subsegments, each of which is mapped bijectively onto a segment in $T$. Given a morphism $f \colon T_0 \to T$, there is a canonical way of constructing \emph{intermediate trees} $T_t$ for $0 \leq t \leq \infty$ with $T_{\infty} = T$ (see \cite{ Sko:Deformationslengthfunctions}, \cite{GL:OuterSpaceFreeProduct}). The tree $T_t$ depends continuously on $f$ and $t$ in the Gromov--Hausdorff topology. There exist morphisms $\phi_t\colon T_0 \to T_t$ and $\psi_t \colon T_t \to T$ with $\psi_t \circ \phi_t = f$. In particular, this implies that if $T_0$ and $T$ are in $\D$, then so is $T_t$.

In order to prove contractibility of $\mcP\bar{\D}$, Guirardel--Levitt define a map $\rho \colon \bar{\D} \times [0, \infty] \to \bar{\D}$ as follows: 
Fix $T_0 \in \D$ with a minimal number of edge orbits---this means that every vertex has non-trivial stabiliser---and let $\mcC_0$ be the closed simplex in $\bar{\D}$ containing $T_0$. Then $\mcC_0$ is contractible. They associate to $T \in \bar{\D}$ a morphism $f_T \colon T_0(T) \to T$, where $T_0(T)$ is a metric tree in $\mcC_0$. Skora's deformation provides intermediate trees $T_t(T)$. Set $\rho(T, t) := T_t(T)$. Then $\rho(T, \infty) = T$ and $\rho(\bar{\D} \times \{0\}) \subseteq \mcC_0$.  Guirardel--Levitt show that under the choices they make, $\rho$ is continuous with respect to the Gromov--Hausdorff topology.

\subsection{Contractibility of $\FreeS(\ffa)$}

The relative free splitting complex $\real{\FreeS(\ffa)}$ can be seen as a subset of $\mcP \mathcal{T}$. The usual simplicial topology on $\real{\FreeS(\ffa)}$ agrees with the subspace topology it inherits from $\mcP\mathcal{T}$, equipped with the weak topology. This weak topology on $\real{\FreeS(\ffa)}$ restricts to the weak topology on $\mcP \D$. 

To prove that $\FreeS(\ffa)$ is contractible, we now show that $\rho$ restricts to a map on $\real{\FreeS(\ffa)}$ and that the restriction is continuous in the weak topology. 

\begin{theorem}
\label{thm:contractibilty of relative free splitting}
Let $\rank \geq 2$.
For all free factor systems $\ffa$ in $\free$, the poset of free splittings $\FreeS(\ffa)$ of $\free$ relative to $\ffa$ is contractible.
\end{theorem}

\begin{proof}
Up to projectivising, $\FreeS(\ffa)$ is a subset of $\bar{\D}$. Let  $\rho \colon \real{\FreeS(\ffa)} \times [0, \infty] \to \bar{\D}$ be the restriction of the map defined in \cite{GL:DeformationSpace}. We will show that the image of $\rho$ is contained in $\real{\FreeS(\ffa)}$ and that the map is continuous with respect to the weak topology on $\real{\FreeS(\ffa)}$. 
The weak topology agrees with the simplicial topology on $\real{\FreeS(\ffa)}$, so it is sufficient to prove the continuity of the restrictions $\rho_{\sigma} \colon \sigma \times [0, \infty] \to \real{\FreeS(\ffa)}$, where $\sigma$ is a closed simplex of $\real{\FreeS(\ffa)}$. 
 
We first show that for $T \in \real{\FreeS(\ffa)}$, the intermediate trees $T_t(T)$ are also in $\real{\FreeS(\ffa)}$. 
Let $f_T \colon T_0(T) \to T$ be the morphism between $T_0(T), T \in \real{\FreeS(\ffa)}$ as defined in \cite{GL:DeformationSpace}. We have morphisms $\phi_t\colon T_0(T) \to T_t$ and $\psi_t \colon T_t \to T$ with $\psi_t \circ \phi_t = f_T$, so \cite[Lemma~4.3]{GL:OuterSpaceFreeProduct} implies that $T_t(T)$ is a simplicial tree. Since $\phi_t$ is a morphism, in particular equivariant, we have that $\ffa \sqsubseteq \V(T_t(T))$. Since $\psi_t$ is also a morphism and $T$ has trivial edge stabilisers, the same is true for $T_t(T)$. Therefore, $T_t(T) \in \real{\FreeS(\ffa)}$.    

Another consequence of \cite[Lemma~4.3]{GL:OuterSpaceFreeProduct} is that there are only finitely many possibilities for the intermediate tree $T_t(T)$, up to equivariant homeomorphism. This implies that for any closed simplex $\sigma$ in $\real{\FreeS(\ffa)}$, the set of intermediate trees $T_t(T) = \rho(T,t)$, for $t \geq 0$ and $T \in \sigma$ is contained in a finite union of simplices of $\real{\FreeS(\ffa)}$.  
On any finite union of simplices in $\real{\FreeS(\ffa)}$, a set is open with respect to the Gromov--Hausdorff topology if and only if it is open with respect to the weak topology (see \cite[Proposition~5.2]{GL:DeformationSpace} and the remark below it). Hence, the continuity of $\rho_{\sigma}$ with respect to the weak topology on $\real{\FreeS(\ffa)}$ follows from the fact that $\rho_{\sigma}$ is continuous in the Gromov--Hausdorff topology (\cite[Corollary~6.3]{GL:DeformationSpace} applied to $\real{\FreeS(\ffa)}$).
\end{proof}

\subsection{Contractibility of $\FreeSone([A])$}
\label{sec:FreeSone contractibility}
A similar argument as in the proof of Theorem~\ref{thm:contractibilty of relative free splitting} does not quite work to prove contractibility of $\FreeSone([A])$ (see Remark~\ref{rem:counterexample}). Instead, in this section we write $\FreeSone([A])$ as a union of subposets $X(A, A_1, \ldots, A_m : \free)$, defined below, each of which is contractible. We then use the nerve of this covering to prove contractibility of $\FreeSone([A])$.

Let $G$ be a graph in $L$, the spine of $\CV$, and $B$ a finitely generated subgroup of $\free$. 

\begin{definition}[$B|G$]
In \cite{BF:FreeFactorComplex}, Bestvina and Feighn define $B|G$ to be the core of the covering space of $G$ corresponding to $[B]$. There is a canonical immersion from $B|G$ into $G$ which gives $B|G$ a marking. We say \emph{$G$ has a subgraph with fundamental group $B$} if $B|G \hookrightarrow G$ is an embedding. \end{definition}
For example, take $G$ to be a rose with 3 petals and labels $a,\, b,\, c$. Consider the subgroups $B_1 = \la abc \ra$ and $B_2 = \la a, b \ra$ of $\free_3 = \la a, b, c\ra$. Then $B_2|G \hookrightarrow G$ is an embedding and we say $G$ has a subgraph with fundamental group $B_2$. But $B_1|G \hookrightarrow G$ is an immersion that is not an embedding. See Figure~\ref{fig:B|G}. 

\begin{figure}[h]
  \centering{ 
   \def\svgscale{.8}
    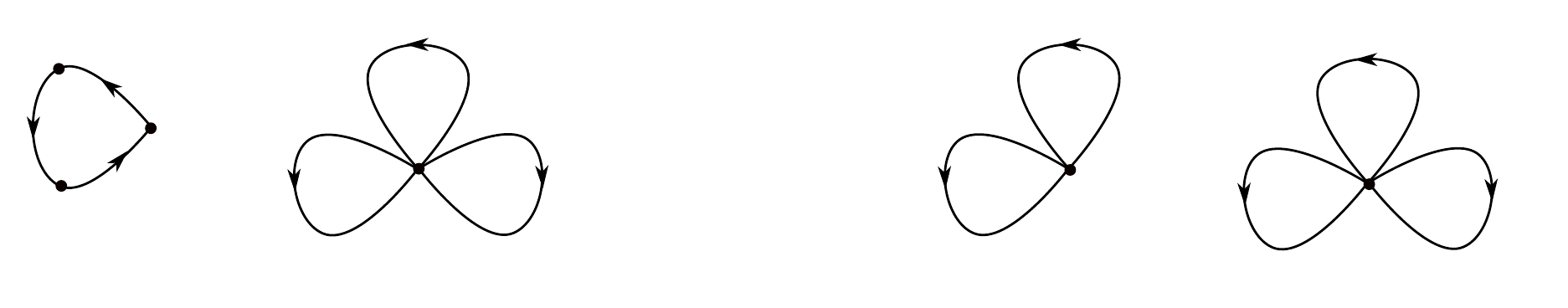
    \caption{Examples of $B|G$ for $G \in L$. Here $\free_3 = \la a, b, c \ra$, $B_1 = \la abc \ra$ and $B_2= \la a,b \ra$.}
    \label{fig:B|G}}
\end{figure}

We extend the above definition to any free splitting $S \in \FreeS$. 

\begin{definition}[$B|S$] \label{defn:B|S} Let $B$ be a proper free factor of $\free$ such that for every $V \in \V(S)$, the intersection $[B] \cap V$ is either trivial or equal to $V$.  Consider a blow-up $\hat{S} \in L$ of $S$ obtained by blowing up all the vertex groups of $S$ to roses. Then there is an immersion $B|\hat{S} \hookrightarrow \hat{S}$. We say \emph{$S$ has a subgraph with fundamental group $B$} or \emph{$B|S$ is a subgraph of $S$} if $B|\hat{S} \hookrightarrow \hat{S}$ is an embedding for (some such) $\hat{S} \in L$. Define $B|S$ to be the graph obtained by collapsing in $B|\hat{S}$ the roses corresponding to each $V \in \V(S)$ contained in $[B]$.  \end{definition}  

In order to see that $B|S$ is well-defined, consider two splittings $\hat{S}$ and $\hat{S}'$ defined as in Definition~\ref{defn:B|S}. Such splittings can only differ in the roses corresponding to the vertex stabilisers of $S$. Thus collapsing the roses for $V \in \V(S)$ in $B|\hat{S}$ and $B|\hat{S}'$ yields the same graph $B|S$. See Figure~\ref{fig:B|S} for some examples of $B|S$. 

\begin{figure}[h]
  \centering{ 
   \def\svgscale{1}
    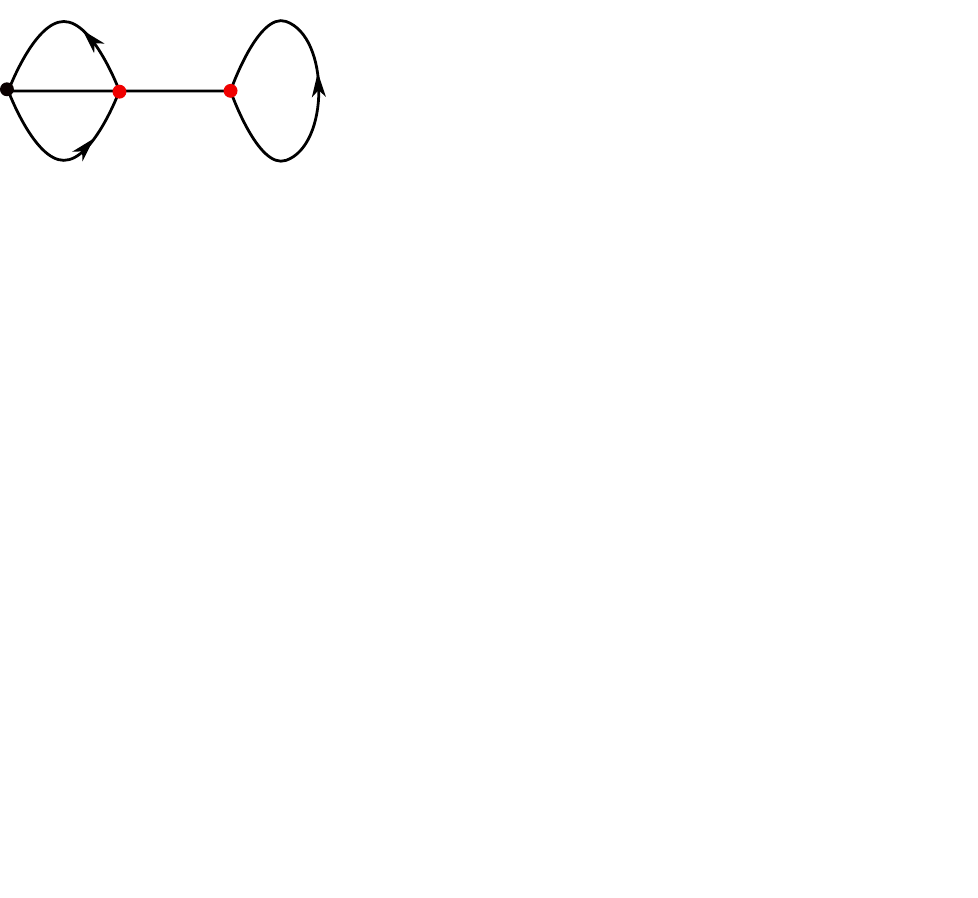
    \caption{Examples of $B|S$ for $S \in \FreeS$. Here $\free_7 = \la a, b, c, d, e, f \ra$, $B_1 = \la d, a, b \ra$, $B_2= \la d, c \ra$ and $B_3 = \la d, f \ra$.}
    \label{fig:B|S}}
\end{figure}

\begin{definition}\label{defn:Poset for free factor systems} For a chain of free factors of $\free$ given by $ A \subset A_1 \subset \ldots \subset A_l \subset B_{0} \subset \ldots \subset B_m  $,
\begin{itemize}
 \item let \emph{$X(A : B_{0}, \ldots, B_m)$} be the poset of all free splittings $S$ such that $\V(S)$ is equal to $[A]$ and $B_i|S$ is a subgraph of $S$ for every $0 \leq i \leq m$;
 \item let \emph{$X(A, A_1, \ldots, A_l : B_{0}, \ldots, B_m)$} be the poset of all free splittings $S$ such that one has $\V(S) \in\nolinebreak \{[A], [A_0], \ldots, [A_l]\}$ and $B_i|S$ is a subgraph of $S$ for every $0 \leq i \leq m$.
\end{itemize}
\end{definition}

\begin{lemma}\label{lem:X(A:A1toAm)}
$X(A : B_0, \ldots, B_m)$ is contractible. 
\end{lemma}
\begin{proof}
Since $\real{X(A : B_0, \ldots, B_m)}$ is a subcomplex of $\mcP\D$, the projectivised Outer space relative to $[A]$, we can consider the restriction $\rho\colon \real{X(A : B_0, \ldots, B_m)} \times [0, \infty] \to \mcP\D$ of the map $\rho$ defined in Section~\ref{sec:Skora}. Let $T_0 \in X(A : B_0, \ldots, B_m)$ be a tree with minimal number of edge orbits. For any $T \in X(A : B_0, \ldots, B_m)$, let $f_T\colon T_0(T) \to T$ be the morphism as defined in \cite{GL:OuterSpaceFreeProduct}.  
As in the proof of Theorem~\ref{thm:contractibilty of relative free splitting}, we only need to show that the intermediate trees $T_t(T)$ are also in $X(A : B_0, \ldots, B_m)$. 
For $0 \leq i \leq m$, both $T_0(T)$ and $T$ have a subgraph with fundamental group $B_i$. By equivariance, the morphism $f_T$ restricts to a morphism between the lifts of $B_i|T_0(T)$ in $T_0(T)$ and $B_i|T$ in $T$.  
The map $f_T$ folds paths in $T_0(T)$ that have the same image in $T$. By definition, $T_t(T)$ is obtained by performing only folds of paths up to length $t$ (see \cite[Section 3.1]{GL:OuterSpaceFreeProduct}). As $f_T$ maps $B_i|T_0(T)$ to $B_i|T$, it follows that $T_t(T)$ also has a subgraph with fundamental group $B_i$. Thus $T_t(T) \in X(A : B_0, \ldots, B_m)$.    
\end{proof}

\begin{lemma}
\label{lem:X(B0toBl:A0toAm) contractible} 
$X(A, A_1, \ldots, A_l : B_0, \ldots, B_m)$ is contractible. In particular, $X(A, A_1, \ldots, A_l : \free)$ is contractible. 
\end{lemma}
\begin{proof}
The proof is by induction on $l$. By Lemma~\ref{lem:X(A:A1toAm)}, the claim holds true for all $m$ if $l=0$. Now assume that it holds true up to $l-1$.

Then in particular, the posets 
\begin{gather*}
X_{l-1} \coloneqq X(A, A_1, \ldots, A_{l-1} : B_0, \ldots, B_m),\hspace{0.5 cm}  X_{l-1,l} \coloneqq X(A, A_1, \ldots, A_{l-1} : A_l, B_0, \ldots, B_m)\\
 \text{and} \hspace{0.5 cm} X_l \coloneqq X(A_l:B_0, \ldots, B_m)
\end{gather*}
are contractible. By definition $X_{l-1,l}$ is the subposet of $X_{l-1}$ consisting of all those $S \in X_{l-1}$ that collapse to some free splitting in $X_l$. For each such $S \in X_{l-1}$, there is a unique maximal splitting $S' \in X_l$, on which $S$ collapses, namely the subgraph $S' = S/(A_l|S)$ obtained by collapsing $A_l|S$.  Hence, the map
\begin{align*}
X_l\cup X_{l-1,l} &\to X_l\\
S&\mapsto
\begin{cases}
S' &,\, S\in X_{l-1,l},\\
S &,\, S\in X_l,
\end{cases}
\end{align*}
induces a deformation retraction $\real{X_{l-1,l}\cup X_l}\to \real{X_l}$.

It follows that $\real{X(A, A_1, \ldots, A_l : B_0, \ldots, B_m)}  = \real{X_{l-1} \cup X_l}$ is obtained by gluing together $\real{X_{l-1}}$ and $\real{X_l}$ along $\real{X_{l-1,l}}$. Now $\real{X_{l-1}}, \real{X_{l-1,l}}$ and $\real{X_l}$ are contractible by assumption, whence the claim follows.
\end{proof}

\begin{remark}\label{rem:counterexample}
We would like to remark that Lemma~\ref{lem:X(B0toBl:A0toAm) contractible} cannot be proved using Skora paths as was done in the proof of Lemma~\ref{lem:X(A:A1toAm)}. It is possible to construct Skora paths between $T_0, T \in X(A, A_1:\free)$ such that there exists an intermediate tree $T_t$ with $\V(T_t) \notin \{[A], [A_1]\}$. For the same reason, Skora paths are not useful to show contractibilty of $\FreeSone([A])$ directly. 
\end{remark}

We are now ready to prove that $\FreeSone([A])$ is contractible. 

\begin{theorem}
\label{thm:contractibilty of FreeSone}
Let $\rank \geq 2$.
For all proper free factors $A$ in $\free$, the poset $\FreeSone_\rank([A])$ consisting of all free splittings having exactly one non-trivial vertex group $[B]\geq [A]$ is contractible.
\end{theorem}
\begin{proof}
Each simplex $\sigma$ in the order complex $\Delta(\FreeSone(A))$ is of the form $S_0\to \ldots \to S_k$ where each $S_i$ is a free splitting of $\free$ collapsing to $S_{i+1}$. Furthermore, the vertex group systems of these free splittings form a chain $\V(S_0) \leq \ldots \leq \V(S_k)$ of free factor systems such that $[A] \leq \V(S_i)$ for all $i$.
It follows that $\sigma$ is contained in $\Delta(X(A, \V(S_0), \ldots  , \V(S_k):\free))$. Hence the realisation $\real{\FreeSone([A])}$ can be written as a union
\begin{equation*}
\real{\FreeSone([A])}=\bigcup_{A \subset A_1 \subset \ldots  \subset A_l} \real{X(A, A_1, \ldots  , A_l :\free)}.
\end{equation*}

By Lemma~\ref{lem:X(B0toBl:A0toAm) contractible}, each $\real{X(A, A_1, \ldots  , A_l :\free)}$ is contractible. Furthermore, one has
\begin{equation*}
\real{X(A, A_1, \ldots  , A_l :\free)} \cap \real{X(A, B_1, \ldots  , B_m :\free)}= \real{X(A, C_1, \ldots  , C_k :\free)}
\end{equation*}
where $[A] < [C_1] < \ldots  < [C_k]$ is the longest common subchain of $[A] < [A_1]<  \ldots  < [A_l]$ and $[A]< [B_1]<  \ldots  <[B_m]$. Consequently, all intersections of these sets are contractible and Lemma~\ref{lem:NerveTheorem} implies that $\real{\FreeSone([A])}$ is homotopy equivalent to the nerve of this covering. However, as all of these sets contain $\real{X(A:\free)}$, they intersect non-trivially, so this nerve complex is contractible.
\end{proof} 

\section{Factor complexes at infinity}\label{sec:homeq}
\label{sec:complexes at infinitiy}

In this short section, we connect the factor complexes considered so far to subposets of $\FreeS$ which sit at the boundary of the simplicial completion of $\CV$. Let $\FreeSone_{\rank}$ be the subposet of $\FreeS$ given by free splittings that have exactly one non-trivial vertex group.
For fixed $\rank\geq 2$, we define
\begin{align*}
\bFreeS\coloneqq \FreeS \bsl L,&& \FreeSone\coloneqq\FreeSone_{\rank}, && \bFreeSred\coloneqq \FreeSred \bsl K.
\end{align*}

The next proposition follows almost immediately from the contractibility of the relative free splitting complexes established in the preceding section.

\begin{proposition}\label{prop:HomotopyEquivalence} 
Let $\rank \geq 2$.
\begin{enumerate} 
\item $\bFreeS$ is homotopy equivalent to $\FFS$. 
\item $\FreeSone$ is homotopy equivalent to $\FreeF$. 
\end{enumerate}
\end{proposition}
\begin{proof}
Assigning to each splitting $S \in \bFreeS$ the free factor system $\V(S)$ given by its non-trivial vertex stabilisers defines a poset map $f\colon \bFreeS \to \FFS^{op}$. For any poset $P$, there is a natural isomorphism of the order complexes $\Delta P$ and $\Delta P^{op}$. Hence, $\Delta(\FFS^{op})\cong \Delta(\FFS)$ which allows us to interpret $f$ as an order-inverting map $f\colon \bFreeS \to\nolinebreak \FFS$.

For any free factor system $\ffa$ in $\free$, the fibre $f^{-1}((\FFS)_{\geq \ffa})$ is equal to the poset $\FreeS(\ffa)$ of free splittings relative to $\ffa$. This poset is contractible by Theorem~\ref{thm:contractibilty of relative free splitting}.

The image $f(\FreeSone)$ is equal to $\FreeF$, so we can consider its restriction $g:\FreeSone\to \FreeF$. Now for any conjugacy class $[A]$ of free factors in $\free$, the preimage $g^{-1}((\FreeF)_{\geq [A]})$ is given by the intersection $\FreeS([A]) \cap \FreeSone = \FreeSone_\rank([A])$, so Theorem~\ref{thm:contractibilty of FreeSone} finishes the proof.
\end{proof}

\begin{remark}
The map $f\colon \bFreeS \to \FFS$ defined in the proof of Proposition~\ref{prop:HomotopyEquivalence} has already been used to study the geometry of the complexes in question:

In \cite[Section 6.2]{HM:RelativeFree}, the authors define ``projection maps'' $\pi:\FreeS\to\FFS$ and show that these maps are Lipschitz with respect to the metrics on the 1-skeleta of $\FreeS$ and $\FFS$ assigning length $1$ to each edge. The map $f$ can be seen as the restriction of such a projection map to $\bFreeS$ and hence is Lipschitz as well.

Using the language of sphere systems (see Section~\ref{sec:spheresystems}), Hilion and Horbez in \cite[Section 8]{HH:SurgeryPaths} consider the poset $\FreeS^c\subset\FreeSone$ of all free splittings whose corresponding graph of groups is a rose with non-trivial vertex group, i.e. those free splittings of $\FreeSone$ having only one orbit of vertices. They show that the inclusion $\FreeS^c\subset\FreeSone$ defines a quasi-isometry of the 1-skeleta and that the restriction $f:\FreeS^c\to \FreeF$ has quasi-convex fibres. This is used to deduce hyperbolicity of $\FreeF$.
\end{remark}

\section{Higher connectivity of factor complexes}\label{sec:productspace}

In this section, we will combine the results obtained so far in order to establish higher connectivity properties of the various complexes defined in the introduction. Fix $\rank\geq 2$ and define $\bFreeS$, $\FreeSone$ and $\bFreeSred$ as in Section~\ref{sec:complexes at infinitiy}.

\label{sec:freefactorsystems}
Let $L\times\bFreeS$ denote the product poset of the spine of (unreduced) Outer space and its simplicial boundary.
We define $\mcZ$ to be the subposet of $L\times\bFreeS$ given by all pairs $(G,S)$ such that $G\in L$  and $S=G/H$ is obtained by collapsing a proper core subgraph $H\subset G$. Let $p_1:\mcZ\to L$ and $p_2:\mcZ\to \bFreeS$ be the natural projection maps.

We want to use $\mcZ$ to study the connectivity properties of $\bFreeS$. The methods we use for this can also be applied to understand the topology of the free factor complex and the boundary of jewel space. So we will in fact prove connectivity results for all these complexes at the same time. For this we need to introduce two subposets of $\mcZ$:

For the first one, we set $\mcZ^1$ to be the subposet of $L\times\FreeSone$ given by all pairs $(G,S)$ such that $S=G/H$ is obtained by collapsing a proper \emph{connected} core subgraph $H\subset G$. Let $q_1:\mcZ^1\to L$ and $q_2:\mcZ^1\to \FreeSone$ be the natural projection maps. The poset $\mcZ^1$ is a subposet of $\mcZ$ and $q_1$ and $q_2$ are the restrictions of the projections $p_1$ and $p_2$. 

Secondly, we define $\mcZ^r$ to be the subposet of $K\times\bFreeSred$ given by all pairs $(G,S)$ such that $S=G/H$ is obtained by collapsing a proper core subgraph $H\subset G$. 
Note that if a graph $G$ does not contain a separating edge, neither does $G/H$ for any subgraph $H\subset G$. It follows that $\mcZ^r$ is the subposet of $\mcZ$ consisting of all $(G,S)$ such that $G\in K$. 
The natural projection maps $r_1:\mcZ^r\to K$ and $r_2:\mcZ^r\to \bFreeSred$ are obtained by restricting the maps $p_1$ and $p_2$.

We think of $\mcZ$, $\mcZ^1$ and $\mcZ^r$ as \emph{thickened versions} of $\bFreeS$, $\FreeSone$ and $\bFreeSred$, respectively.
In order to deduce connectivity results about these three complexes, we proceed in two steps: First we show that the projections $p_2$, $q_2$ and $r_2$ to the second factors define homotopy equivalences; then we apply the results of Section~\ref{sec:graphposet} to understand the fibres of the projections $p_1$, $q_1$ and $r_1$. 

\subsection{Projections to the second factor}
We first deformation retract the fibres of $p_2$, $q_2$ and $r_2$ to simpler subposets:

\begin{lemma}
\label{lem:p_2^-1FFS}
\leavevmode
\begin{enumerate}
\item For all $S\in \bFreeS$, the fibre $p_2^{-1}(\bFreeS_{\geq S})$ deformation retracts to $p_2^{-1}(S)$.
\item For all $S\in \FreeSone$, the fibre $q_2^{-1}(\FreeSone_{\geq S})$ deformation retracts to $q_2^{-1}(S)$.
\item For all $S\in \bFreeSred$, the fibre $r_2^{-1}(\bFreeSred_{\geq S})$ deformation retracts to $r_2^{-1}(S)$.
\end{enumerate}
\end{lemma}
\begin{proof}
We define a map $f:p_2^{-1}(\bFreeS_{\geq S})\to p_2^{-1}(S)$ as follows: If $(G',S')\in p_2^{-1}(\bFreeS_{\geq S})$, then there are collapse maps $G'\to S'$ and $S'\to S$. Concatenating these maps, we see that $S$ is obtained from $G'$ by collapsing a subgraph $H'\subset G'$. As $S\in \bFreeS =\FreeS\bsl L$, the graph $H'$ has non-trivial fundamental group and can be written as the union of a (possibly trivial) forest $T'$ and its (non-trivial) core $\mathring{H}'$. We define a map 
\begin{align*}
f\colon p_2^{-1}(\bFreeS_{\geq S}) &\to p_2^{-1}(S)\\
(G',S') &\mapsto (G'/T',S).
\end{align*} 
As $S=(G'/T')/\mathring{H}'$, the image $(G'/T',S)$ is indeed an element of $p_2^{-1}(S)$.

Now assume we have $(G'',S'')\geq (G',S')$ in $p_2^{-1}(\bFreeS_{\geq S})$. We obtain a diagram of collapse maps
\begin{equation*}
\xy
\xymatrix{
 G'' \ar[d]_{c} \ar[r]
&  S'' \ar[d] \ar[rd] \\
 G' \ar[r] & S' \ar[r] & S \, .}
\endxy 
\end{equation*}
In general, if there are free splittings $S_1, S_2 \in \FreeS$ such that there is a collapse map $S_1 \to S_2$, then this map is unique up to equivariant isomorphism. Hence, the diagram above commutes and the collapse map $c:G''\to G'$ restricts to a surjection $H''\to H'$. The graph $H'$ is obtained from $H''$ by collapsing a forest in $H''$ and identifying end points of edges of $G''$ that are collapsed by $G''\to G'$. It follows that $H''\to H'$ restricts to a map $\mathring{H''}\to\mathring{H'}$. Hence, we have $c(T'')\supseteq T' $ which implies $G''/T'' \geq G'/T'$. Consequently $f:p_2^{-1}(\bFreeS_{\geq S})\to p_2^{-1}(S)$ is a well-defined, monotone poset map restricting to the identity on $p_2^{-1}(S)$. Now one can use Lemma~\ref{lem:HomotopicPosetMaps} and the usual arguments to see that it defines a deformation retraction.

For (2), if $(G',S')\in q_2^{-1}(\FreeSone_{\geq S})$, then the splitting $S'$ is obtained from $G'$ by collapsing a connected core subgraph and $S$ is obtained from $S'$ by collapsing a subgraph. Concatenating these two collapse maps, one sees that $S=G'/H'$ for a subgraph $H'\subset G'$. This subgraph may be disconnected, but only one of its components has non-trivial fundamental group. It follows that its core $\mathring{H'}$ is connected.
Using this observation, the map $f:p_2^{-1}(\bFreeS_{\geq S})\to p_2^{-1}(S)$  restricts to a monotone poset map $q_2^{-1}(\FreeSone_{\geq S})\to q_2^{-1}(S)$. So the second claim follows from Lemma~\ref{lem:HomotopicPosetMaps} as well.

For (3), recall that if a graph $G$ does not contain a separating edge, then neither does $G/H$ for any subgraph $H\subset G$.
It follows that for all $S\in \bFreeSred$, the map $f$ also restricts to a deformation retraction $r_2^{-1}(\bFreeSred_{\geq S})\to r_2^{-1}(S)$.
\end{proof}

Hence, instead of studying arbitrary fibres, it suffices to consider the preimages of single vertices. We start by using the methods from Section~\ref{sec:posets of free splittings} to show:

\begin{proposition}
\label{prop:p_2^-1(S)FFS}
For all $S\in \bFreeS$, the preimage $p_2^{-1}(S)$ is contractible.
\end{proposition}
\begin{proof}
Fix a free splitting $S\in \bFreeS$ and let $[A_1],\ldots, [A_k]$ be the components of $\V(S)$. Every element in $p_2^{-1}(S)$ is given by blowing up the vertices of $S$ with non-trivial vertex stabiliser. That is, every element in $p_2^{-1}(S)$ is a pair $(G,S)$ such that there is a unique core subgraph $H\subset G$ having connected components $H_1,\ldots, H_k$ where $\pi_1(H_i)=[A_i]$ and $S=G/H$. Fix an element $(G_0, S) \in p_2^{-1}(S)$ such that each $H^0_i$ is a rose.

For every element in $p_2^{-1}(S)$, consider the open simplex in $\CV$ corresponding to it. Let $X \subset \CV$ be the union of all such simplices. We will show that $X$ is contractible using Skora's deformation paths. Since $p_2^{-1}(S)$ is a deformation retract of $X$, this implies that $p_2^{-1}(S)$ is also contractible. 

We define a map $\rho \colon X \times [0, \infty] \to \CV$ as follows (see Section~\ref{sec:freesplittingposets}  and \cite[Section 6]{GL:DeformationSpace} for more details):
Fix $T_0 \in X$ such that $T_0$ is the universal cover of $G_0$ and let $\mcC_0$ be the open simplex of $\CV$ corresponding to $T_0$. Consider a tree $T \in X$. Both $T_0$ and $T$ are obtained from $S$ by blowing up the vertices with stabilisers in $\V(S)$. Hence, there exists a unique $T_0(T) \in \mcC_0$ and a morphism $f_T \colon T_0(T) \to T$, such that $f_T$ maps the minimal subtree stabilised by $A_i \in [A_i]$ in $T_0(T)$ to the minimal subtree of $A_i \in [A_i]$ in $T$. 
As in Section~\ref{sec:freesplittingposets}, the morphism $f_T$ gives rise to a sequence of intermediate trees $T_t(T)$. We set $\rho(T,t) \coloneqq T_t(T)$.

Since $f_T$ restricts to a morphism between the corresponding minimal subtrees of $A_i$ and it factors through the morphisms $T_0(T)\to T_t(T)$ and $T_t(T) \to T$, the intermediate tree $T_t(T)$ collapses to $S$ and hence is in $X$. Therefore, we have a map $\rho \colon X \times [0, \infty] \to X$. The same arguments as in proof of \cite[Theorem~6.1(1)]{GL:DeformationSpace} immediately give that $\rho$ is continuous with respect to the weak topology on $X$. This implies that $X$ is contractible. 
\end{proof}

The following shows that Proposition~\ref{prop:p_2^-1(S)FFS} also provides us with sufficient information about the fibres of $q_2$ and $r_2$.
\begin{proposition}
\label{prop:p_2^-1(S) to q_2^-1(S) and r_2^-1(S)}
\leavevmode
\begin{enumerate}
\item For all $S\in \FreeSone$, one has $q_2^{-1}(S)= p_2^{-1}(S)$.
\item For all $S\in \bFreeSred$, there is a deformation retraction $p_2^{-1}(S)\to r_2^{-1}(S)$.
\end{enumerate}
\end{proposition}
\begin{proof}
The first claim follows immediately from the definitions: The map $q_2$ is the restriction of $p_2:\mcZ\to\bFreeS$.

The proof of the second claim essentially just uses that $K$ is a deformation retract of $L$:
Let $S\in \bFreeSred$.
By definition, $p_2^{-1}(S)$ can be seen as the poset of all $G\in L$ such that there is a proper core subgraph $H\subset G$ with $G/H=S$. On the other hand, $r_2^{-1}(S)$ consists of all $G\in K$ satisfying the same property, so $r_2^{-1}(S)=p_2^{-1}(S)\cap K$ is a subposet of $p_2^{-1}(S)$.

For each $G\in L$, there is a unique maximal $G'\in K$ such that $G'\leq G$; it is obtained by collapsing all the separating edges of $G$.
Now if $G\in p_2^{-1}(S)$ and $H\subset G$ is a core subgraph such that $S=G/H$, the collapse $G\to G'$ maps $H$ to a core subgraph $H'\subset G'$. Because $S$ does not contain any separating edges, the collapse $G\to S$ factors as
\begin{equation*}
\xy
\xymatrix{
G \ar[rrd]^{\text{collapse }H} \ar[d]
& \\
G' \ar[rr]_{\text{collapse }H'} && S .}
\endxy
\end{equation*}
It follows that $G'\in r_2^{-1}(S)$, so the assignment $G\mapsto G'$ defines a poset map $p_2^{-1}(S)\to r_2^{-1}(S)$ that is monotone and restricts to the identity on $r_2^{-1}(S)$. The statement now follows from Lemma~\ref{lem:HomotopicPosetMaps}.
\end{proof}

In particular, these fibres are contractible.

\begin{corollary}
\label{cor:ZhomeqbFree}
The maps $p_2:\mcZ\to \bFreeS$, $q_2:\mcZ^1\to \FreeSone$ and $r_2:\mcZ^r \to \bFreeSred$ are homotopy equivalences.
\end{corollary}
\begin{proof}
Using Quillen's fibre lemma, the claim is an immediate consequence of Lemma~\ref{lem:p_2^-1FFS}, Proposition~\ref{prop:p_2^-1(S)FFS} and Proposition~\ref{prop:p_2^-1(S) to q_2^-1(S) and r_2^-1(S)}.
\end{proof}

\subsection{Projections to the first factor} 
Corollary~\ref{cor:ZhomeqbFree} allows us to replace $\bFreeS$ by its thickened version $\mcZ$. 
This has the advantage that $\mcZ$ possesses a natural projection map $p_1$ to the contractible poset $L$ which we will study in this subsection.

\begin{lemma}
\label{lem:Fibres&CoreSubgraphs}
For all $G\in L$, the fibre $p_1^{-1}(L_{\leq G})$ is homotopy equivalent to $\Core(G)$, the poset of proper core subgraphs of $G$.
\end{lemma}
\begin{proof}
Each element of the fibre $p_1^{-1}(L_{\leq G})$ consists of a pair $(G',S')$ where $G'\leq G$ in $L$ and $S'\in\bFreeS$ is obtained from $G'$ by collapsing a proper core subgraph $H'$.
As $G'$ is obtained from $G$ by collapsing a forest, there is a unique, proper core subgraph $H$ of $G$ making the following diagram commute:
\begin{equation*}
\xy
\xymatrix{
H \ar[d] \ar@{^{(}->}[r]
& G \ar[d] \\
H' \ar@{^{(}->}[r] & G' }
\endxy
\end{equation*}
$H$ is the unique core subgraph of $G$ such that $\pi_1(H)=\pi_1(H')$.

Because this diagram commutes, the collapse $G\to G'$ induces a collapse $G/H \to G'/H'=S'$. Hence, we get a monotone poset map
\begin{align*}
f:p_1^{-1}(L_{\leq G})&\to p_1^{-1}(G)\\
(G',S')&\mapsto (G,G/H)
\end{align*}
restricting to the identity on $p_1^{-1}(G)\subseteq p_1^{-1}(L_{\leq G})$.
Again Lemma~\ref{lem:HomotopicPosetMaps} implies that $f$ defines a deformation retraction.

If $H$ and $H'$ are proper core subgraphs of $G$, one has $G/H\geq G/H'$ in $\bFreeS$ if and only if $H\leq H'$ in $\Core(G)$. It follows that $p_1^{-1}(G)$ can be identified with $\Core(G)^{op}$. Noting that $\real{\Core(G)^{op}}\cong \real{\Core(G)}$ finishes the proof.
\end{proof}

\begin{lemma}
\label{lem:Fibres&cCoreSubgraphs}
For all $G\in L$, the fibre $q_1^{-1}(L_{\leq G})$ is homotopy equivalent to $\cCore(G)$, the poset of proper connected core subgraphs of $G$.
\end{lemma}
\begin{proof}
The proof is literally the same as the one of Lemma~\ref{lem:Fibres&CoreSubgraphs} after one makes the following observation: Whenever $G'\leq G$ in $L$ and $H'$ is a proper \emph{connected} core subgraph of $G'$, there is a unique, proper  \emph{connected} core subgraph $H\subset G$ making the diagram
\begin{equation*}
\xy
\xymatrix{
H \ar[d] \ar@{^{(}->}[r]
& G \ar[d] \\
H' \ar@{^{(}->}[r] & G' }
\endxy
\end{equation*}
commute. (Here again we use that our core subgraphs are allowed to have separating edges.)
\end{proof}

\begin{theorem}
\label{thm:connectedness bFreeSred}
\label{thm:connectedness FreeSone}
For $\rank\geq 3$, the posets $\bFreeS$, $\FreeSone$ and $\bFreeSred$ are $(\rank-3)$-connected.
\end{theorem}
\begin{proof}
We already know that $\bFreeS$ is homotopy equivalent to $\mcZ$.
To show that $\mcZ$ is $(\rank-3)$-connected, consider the first projection $p_1:\mcZ\to L$. By Lemma~\ref{lem:Fibres&CoreSubgraphs}, the fibre $p_1^{-1}(L_{\leq G})$ is homotopy equivalent to $\Core(G)$ for all $G\in L$. Lemma~\ref{lem:SubgraphstoCores} and Proposition~\ref{prop:X(G) homotopy} imply that this poset is at least $(\rank-3)$-connected. Applying Lemma~\ref{lem:QuillenFibres} finishes the proof.

For $\FreeSone$, the proof is just the same: By Proposition~\ref{cor:ZhomeqbFree}, the poset $\FreeSone$ is homotopy equivalent to $\mcZ^1$. For each $G\in L$, the fibre $q_1^{-1}(L_{\leq G})$ of the projection map $q_1:\mcZ^1\to L$ is homotopy equivalent to $\cCore(G)$ by Lemma~\ref{lem:Fibres&cCoreSubgraphs}. This poset is $(\rank-2)$-spherical by Lemma~\ref{lem:cSubgraphstocCores} and Proposition~\ref{prop:cX(G) homotopy}.

Lastly, the poset $\bFreeSred \simeq \mcZ^r$ is $(\rank-3)$-connected because the fibres of $r_1:\mcZ^r\to K$ are equal to those of $p_1$. Indeed, the spine $K$ of reduced Outer space is a downwards-closed subposet of $L$ and for any $G\in K$, one has $r_1^{-1}(G)=p_1^{-1}(G)$. It follows that for all $G\in K$, one has an equality $r_1^{-1}(K_{\leq G})$ and $p_1^{-1}(L_{\leq G})$. The higher connectivity of $\bFreeSred$ now follows as above.
\end{proof}

For the free factor complex and the boundary of jewel space, the preceding theorem already yields the best connectivity results that we are able to obtain. We summarise them in the following two corollaries.

\begin{corollary}
For all $\rank \geq 2$, the free factor complex $\FreeF$ is homotopy equivalent to a wedge of $(\rank-2)$-spheres.
\end{corollary}
\begin{proof}
If $\rank=2$, the complex $\Delta(\FreeF)$ is a countable infinite set carrying the discrete topology, i.e. a wedge of $0$-spheres. For $\rank \geq 3$,  $\Delta(\FreeF)$ is a simplicial complex of dimension $\rank-2$, so using the Whitehead theorem, it suffices to show that it is $(\rank-3)$-connected and non-contractible. We have $\FreeSone\simeq \FreeF$ by Proposition~\ref{prop:HomotopyEquivalence}, so the complex is $(\rank-3)$-connected by the preceding theorem.

To prove that it is non-contractible, look at the following subposet $\Sigma\subset \FreeF$:
Choose a basis $\set[x_1, x_2, \ldots , x_\rank]$ of $\free$ and let $\Sigma$ be the poset of all conjugacy classes of free factors generated by proper subsets of this basis. It is easy to see that $\real{\Sigma}$ is a triangulated $(\rank-2)$-sphere inside $\real{\FreeF}$. (This subcomplex is the analogue of an apartment in a Tits-building; see also \cite[Section 5]{HV:ComplexOfFreeFactors}.) In particular, this shows that $H_{\rank-2}(\FreeF)$ is non-trivial, so the complex cannot be contractible.
\end{proof}

In the reduced setting, Theorem~\ref{thm:connectedness bFreeSred} and the first statement of Theorem~\ref{thm:BoundaryJewelSpace} immediately imply:
\begin{corollary}
For all $\rank\geq 3$, the boundary $\partial \jewel$ is $(\rank-3)$-connected.
\end{corollary}

For $\bFreeS$, we can further improve the result of Theorem~\ref{thm:connectedness bFreeSred} because the following lemma gives us additional information about the fibres of $p_1$.

\begin{lemma}\label{lem:contractible} For $G \in L$, let $f \colon p_1^{-1}(L_{\leq G}) \to \mcZ$ be the inclusion map. Then the induced map on homotopy groups $f_{\ast} \colon \pi_{\rank-2} (p_1^{-1}(L_{\leq G})) \to \pi_{\rank-2}(\mcZ)$ is trivial. \end{lemma}
\begin{proof}
As $p_1^{-1}(L_{\leq G})$ deformation retracts to $p_1^{-1}(G)$ (see the proof of Lemma \ref{lem:Fibres&CoreSubgraphs}), it is sufficient to show that the map $f_{\ast} \colon \pi_{\rank-2} (p_1^{-1}(G)) \to \pi_{\rank-2}(\mcZ)$ is trivial. In order to prove this, we will describe explicit generators for $\pi_{\rank-2} (p_1^{-1}(G))$ and show that these map to the identity under $f_{\ast}$.

We start by describing $\tilde{H}_{\rank-2}(X(G))$, where $X(G)$ is the poset of all proper subgraphs of $G$ that are non-empty and where at least one component has non-trivial fundamental group: Recall that the poset $\Sub(G)$ of all proper subgraphs of $G$ has a geometric realisation that is homeomorphic to a sphere of dimension $|E(G)|-2 = |V(G)|+\rank -3 $. Hence, by Alexander duality (Lemma \ref{lem:AlexanderDuality}), we have an isomorphism
\begin{equation*}
\tilde{H}_{\rank-2}(X(G)) \cong \tilde{H}^{|V(G)|+\rank -3 -(\rank-2)-1} (\Sub(G)\bsl X(G)) = \tilde{H}^{|V(G)|-2}(\Forest(G)).
\end{equation*}
The collection $\{\sigma^i\}_{i=1}^N$ of $(|V(G)|-2)$-simplices of $\real{\Forest(G)}$ is in bijection with the collection $\{\mathcal{E}_i\}_{i=1}^N$ of maximal forests of $G$. Let $\phi_i$ be the cochain dual to $\sigma^i$, i.e. $\phi_i(\sigma^i) = 1$ and $\phi_i(\sigma^j) = 0$ for $j \neq i$. As $\real{\Forest(G)}$ has dimension $|V(G)|-2$, all of these cochains are cocycles and the corresponding cohomology classes $\{[\phi_i]\}_{i=1}^N$ generate $\tilde{H}^{|V(G)|-2}(\Forest(G))$. 
 Under the isomorphism given by Alexander duality, the class $[\phi_i]$ is identified with the cycle corresponding to the simplex 
whose vertices are all subgraphs of $G-\mathcal{E}_i$. This simplex can be identified with (the order complex of) $\Sub(G/\mathcal{E}_i)$, seen as a subposet of $X(G)$.
Since $\mathcal{E}_i$ is a maximal forest of $G$, the quotient $G/\mathcal{E}_i$ is a rose and we have $\Sub(G/\mathcal{E}_i) = X(G/\mathcal{E}_i)$. We write $\{\real{X(G/\mathcal{E}_i)}\}_{i=1}^N$ for the generating system of $\tilde{H}_{\rank-2}(X(G))$ that we obtain this way. 

This gives us a generating set for $\pi_{\rank-2} (p_1^{-1}(G))$ as follows: As explained in the proof of Lemma \ref{lem:Fibres&CoreSubgraphs}, $p_1^{-1}(G)$ is canonically identified with $C(G)$, the poset of proper core subgraphs of $G$. This poset is a deformation retract of $X(G)$ by Lemma \ref{lem:SubgraphstoCores}, yielding an identification $\pi_{\rank-2} (p_1^{-1}(G)) \cong \pi_{\rank-2}(X(G))$. The involved posets are $(\rank - 3)$-connected by Lemma~\ref{prop:X(G) homotopy}, so for $\rank \geq 4$, the Hurewicz Theorem implies that this is further identified with $H_{\rank-2}(X(G))=\tilde{H}_{\rank-2}(X(G))$. 

After all these identifications, $f_{\ast}$ maps the generator $\real{X(G/\mathcal{E}_i)} \in H_{\rank-2}(X(G))\cong \pi_{\rank-2} (p_1^{-1}(G))$ to the homotopy class of the map $\psi_i: S^{\rank-2}\to \real{\mcZ}$ that sends $S^{\rank-2}$ homeomorphically to the boundary of the simplex whose vertices are all pairs $(G, S)$, where $S$ is obtained by collapsing a subgraph of the rose $G/\mathcal{E}_i$. For all these vertices, we have $(G, S) \geq (G/\mathcal{E}_i, S)$. This implies that $\psi_i$ is homotopic to the map $\tilde{\psi}_i$ whose image is the boundary of the simplex with vertices $(G/\mathcal{E}_i, S)$. 
Because $G/\mathcal{E}_i$ is a rose, there exists $G_i \in L$ such that $G_i$ has a separating edge and $G_i > G/\mathcal{E}_i$.  Now the image of $\tilde{\psi}_i$ is contained in $\real{p_1^{-1}(L_{\leq G_i})}\subseteq \real{\mcZ}$ and this space is contractible by Lemma~\ref{prop:X(G) homotopy}. Hence, $[\tilde{\psi}_i]= f_{\ast}(\real{X(G/\mathcal{E}_i)})$ is trivial.

For $\rank=3$, the lemma follows by an explicit computation.
\end{proof}

\begin{remark} It is possible that in $\mcZ$, the preimage $p_1^{-1}(G)$ has multiple contractions. This can give rise to higher dimensional spheres in $\mcZ$. See Example~\ref{ex:n=2} at the end of the paper. \end{remark}

We are now ready to prove:
\begin{theorem}
\label{thm:connectedness bFreeS}
For $\rank\geq 2$, the poset $\bFreeS$ is $(\rank-2)$-connected.
\end{theorem}
\begin{proof}
By Corollary~\ref{cor:ZhomeqbFree}, $\bFreeS$ is homotopy equivalent to $\mcZ$. The spine $L$ of Outer space is contractible and it follows from Lemma~\ref{lem:Fibres&CoreSubgraphs} and Proposition~\ref{prop:X(G) homotopy} that the fibres of $p_1:\mcZ\to L$ are either $(\rank-3)$-connected or contractible. Using Lemma~\ref{lem:contractible} and applying Lemma~\ref{lem:UpConnectivitiyByOne}, one gets that $\bFreeS$ is $(\rank-2)$-connected. 
\end{proof}

Proposition~\ref{prop:HomotopyEquivalence} immediately implies the following corollary which completes the proof of Theorem~\ref{thm:freefactorsystems}.
\begin{corollary}
The complex $\FFS$ of free factor systems is $(\rank-2)$-connected.
\end{corollary}

Note that in contrast to the situation here, these arguments cannot be used to deduce $(\rank-2)$-connectivity of $\partial \jewel$ as the fibres of the map $r_1$ a priori do not satisfy the conditions needed to apply Lemma~\ref{lem:UpConnectivitiyByOne}. For more comments on the optimality of the result obtained here, see Section~\ref{sec:CV_2}.

\section{Some remarks}
\label{sec:remarks}

\subsection{Relative complexes}
In \cite{HM:RelativeFree}, the authors do not only study the whole poset of free factor systems, but also relative versions of it. For a given free factor system $\ffa$, the \emph{poset of free factor systems of $\free$ relative to $\ffa$} consists of all free factor systems $\ffb$ in $\free$ such that there are proper inclusions $\ffa\sqsubset \ffb \sqsubset \free$. In other words, this poset is given by $(\FFS)_{>\ffa}$.

Replacing $\CV$ by Outer space relative to $\ffa$ one can apply the arguments used in the previous sections in order to show higher connectivity of these relative complexes of free factor systems. As already in the ``absolute'' case, we make use of relative Outer space, most proofs can be taken literally for the relative setting as well. The fibres one obtains and needs to analyse here correspond to posets of graphs with a labelling of the vertices (as described e.g. in \cite{BF:Bordification}). This can be done using similar arguments as in the proof of Proposition~\ref{prop:cX(G) homotopy}.

In his thesis \cite{Bru:buildingsfreefactora}, the first-named author studies these relative versions and obtains analogues of Theorem \ref{thm:FreeFactorComplex}, Theorem \ref{thm:freefactorsystems} and Theorem \ref{thm:FreeSplittingIntroduction}. The results there are proved in the more general setting of automorphism groups of free products. These are used to determine the homotopy type of an analogue of the free factor complex for automorphisms of right-angled Artin groups in \cite{Bru:buildingsfreefactor}.

\subsection{Sphere systems}
\label{sec:spheresystems}
There is an equivalent description of the free splitting complex in terms of sphere systems. For this, let $M_\rank$ be the connected sum of $\rank$ copies of $S^1\times S^2$. This manifold has fundamental group isomorphic to $\free$. A collection $\set[S_1,\ldots ,S_k]$ of disjointly embedded $2$-spheres in $M_\rank$ is called a \emph{sphere system} if no $S_i$ bounds a ball in $M_\rank$ and no two spheres are isotopic. The set of isotopy classes of such sphere systems has a partial order given by inclusion of representatives. The order complex $S(M_\rank)$ of this poset is called the \emph{complex of sphere systems}.
Considering the fundamental group of its complement, each sphere system induces a free splitting of $\free$. In fact, the free splitting complex $\FreeS$ is the barycentric subdivision of $S(M_\rank)$.

Following this, our considerations in this article can be translated in the language of such sphere systems: The complex $\bFreeS$ corresponds to the complex $S_\infty\subset S(M_\rank)$ consisting of all sphere systems $\sigma$ whose complement $M_\rank\bsl\sigma$ has at least one connected component that is not simply-connected.
The complex $\FreeSone$ on the other hand corresponds to $S_1(M_\rank)\subset S(M_\rank)$, the subcomplex of $S(M_\rank)$ consisting of sphere systems whose complement has exactly one component that is not simply-connected.

Using this description, $(\rank-3)$-connectivity of $\bFreeS$ can be deduced very quickly as follows:
\begin{proof}[Proof of $(\rank-3)$-connectivity via sphere systems \cite{V:Private}]
Whenever one takes a sphere system $\sigma$ consisting of at most $(\rank-1)$-many spheres, it induces a free splitting of $\pi_1(M_\rank)\cong \free$ with at most $\rank-1$ orbits of edges. It follows that at least one of the vertex groups of this splitting must be non-trivial, implying that the complement $M_\rank\bsl\sigma$ contains at least one connected component with non-trivial fundamental group. Hence, the entire $(\rank-2)$-skeleton of $S(M_\rank)$ is contained in $S_\infty\cong\bFreeS$. However, the complex $S(M_\rank)$ is contractible (see \cite{Hatcher:FSContractible}), so we have $\set[0]\cong\pi_{\rank-3}(S(M_\rank))\cong\pi_{\rank-3}(S_\infty)$.
\end{proof}
The same argument also shows $(\rank-3)$-connectivity of $\bFreeSred\simeq\partial\jewel$, the second part of Theorem~\ref{thm:BoundaryJewelSpace}.
However, we would like to point out that this does a priori not give a proof for $(\rank-2)$-connectivity of $\bFreeS$ and also in particular does not show connectivity properties of $\FreeSone\simeq \FreeF$.

\subsection{The simplicial boundaries of $\CV[2]$ and $\CVred[2]$}
\label{sec:CV_2}
The difference in the degree of connectivity between the reduced and the unreduced setting might be surprising at first glance, but in fact it can easily be seen when one considers the case where $\rank=2$.

Here, reduced Outer space $\CVred$ can be identified with the tesselation of the hyperbolic plane by the Farey graph (an excellent picture of this tesselation can be found in \cite{V:WhatIs}). The triangles of this tessellation correspond to the three-edge ``theta graph''. Each side of such a triangle is given by graphs that are combinatorially roses with two petals and obtained by collapsing one of the edges of the theta graph; as the rose is a graph of rank $2$, these edges are contained in the interior of $\CVred[2]$. In contrast to that, the vertices of the triangles correspond to loops obtained by collapsing two edges of the theta graph and hence are points sitting at infinity. Hence, the simplicial boundary of $\CVred[2]$ is homeomorphic to $\mathbb{Q}$, a countable join of $0$-spheres.

\begin{figure}
\begin{center}
\includegraphics[scale=0.7]{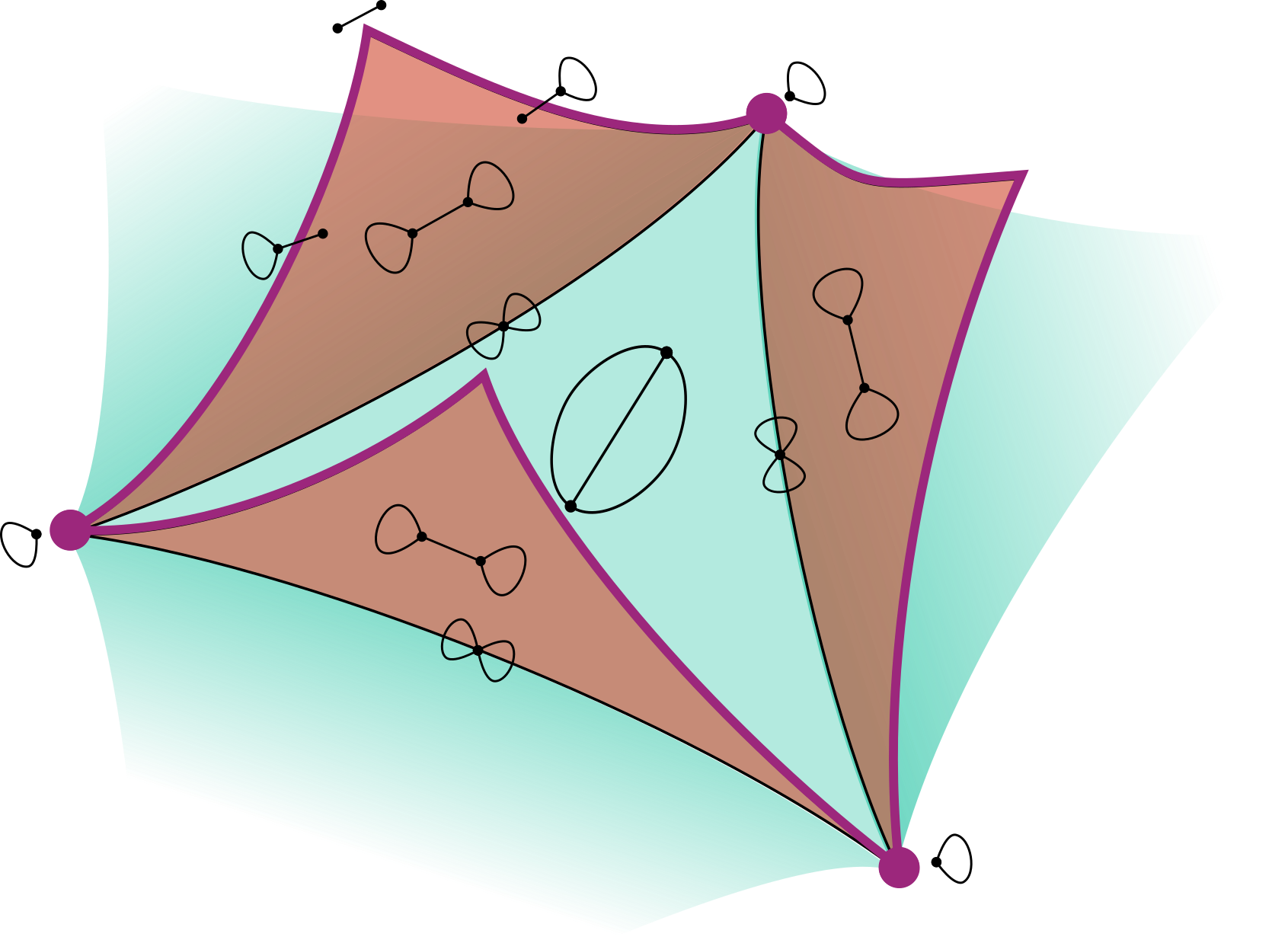}
\end{center}
\caption{A part of $\CV[2]$. The turquoise bottom part is reduced Outer space, together with the red fins on top it forms unreduced Outer space. The faces at infinity are coloured in plum, where the three round vertices are the only points contained in the reduced boundary $\partial_s\CVred$.}
\label{fig:CV_2}
\end{figure}

Starting from reduced Outer space, unreduced $\CV[2]$ is obtained by adding ``fins'' above each edge of the Farey graph. These fins are triangles corresponding to the ``dumbbell graph'' which consists of two loops connected by a separating edge. Collapsing this separating edge, one obtains the side of the triangle that corresponds to the rose. On the other hand, collapsing one of the two loops of the dumbbell yields a graph of rank one, forcing the other two sides of the triangle to sit at infinity. Inside the simplicial boundary $\partial_s\CV[2]$, the concatenation of these sides now connects two vertices of the adjacent theta graph triangles as depicted in Figure~\ref{fig:CV_2}. It follows that  $\partial_s\CV[2]$ is isomorphic to the barycentric subdivision of the Farey graph which is in turn homotopy equivalent to a countable wedge of circles.

This argument answers Question~\ref{qes:HomotopyType?} for $\rank=2$: Here the lower bounds we get for the degree of connectivity of the simplicial boundaries $\partial_s\CVred\simeq\partial \jewel$ and $\partial_s\CV\simeq\FFS$ are optimal and furthermore, the homology of these complexes is concentrated in dimension $\rank-2$ and $\rank-1$, respectively. For higher rank, this is however not clear at all as $\partial \jewel$ and $\FFS$ have dimension $2\rank-3$.
In the case of $\partial \jewel$, there are obvious $(\rank-2)$-spheres one might expect to be non-trivial elements of $\pi_{\rank-2}(\partial \jewel)$. Namely whenever one has an open $(\rank-1)$-simplex in $\CVred$ corresponding to a rose with $\rank$ petals, all of its faces are contained in the simplicial boundary $\partial_s\CVred$. We suspect that the spheres formed by these faces are not contractible inside the boundary but right now we do not see how this could be shown.

\clearpage

\section{Appendix}
The following example illustrates Lemma~\ref{lem:contractible} in the case where $\rank=2$.
\begin{example}\label{ex:n=2}
For $\rank = 2$ and $G \in L$, the preimage $p_1^{-1}(G)$ is either contractible or a wedge of 0-spheres. Suppose $G$ is a theta graph. Then a 0-sphere $s_i$ in $p_1^{-1}(G)$ is isomorphic to $\{G\} \times X(G_i)$, where $G_i$ is a rose obtained from $G$ by collapsing a maximal forest, for $i = 1,2,3$. We claim that each such 0-sphere is contractible in $\mcZ$. Indeed, for each rose $G_i$ consider the dumbbell graph $G_i'$ obtained by blowing up $G_i$ to have a separating edge. Then $p_1^{-1}(G_i')$ is contractible. Now in $\mcZ$, the sphere $s_i$ can be homotoped into $p_1^{-1}(G_i')$. Thus each $s_i$ is contractible in $\mathcal{Z}$. See Figure~\ref{fig:n=2}.   
\end{example}

\begin{figure}[h]
  \centering{ 
   \def\svgscale{.7}
    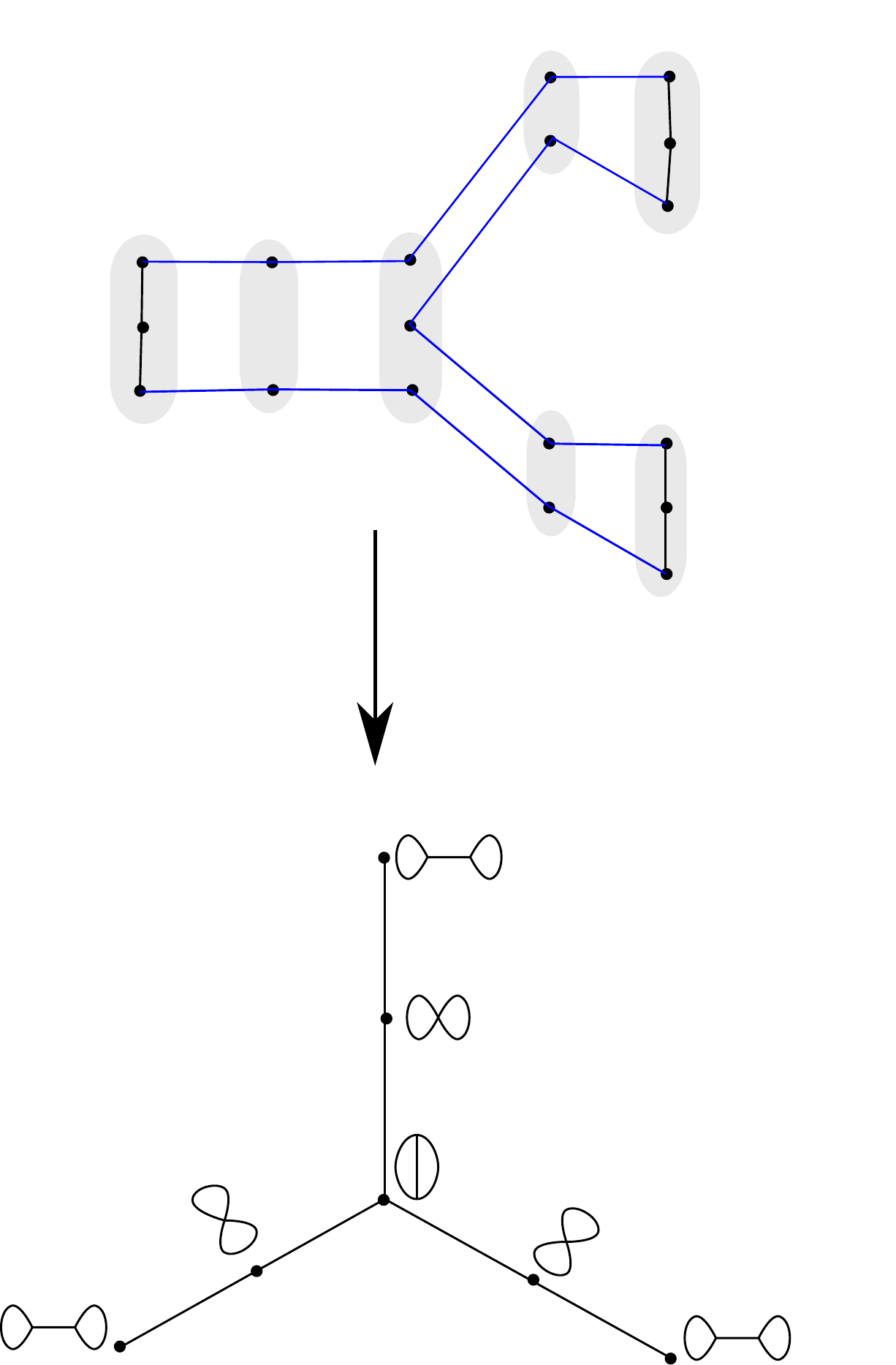
    \caption{Projection map $p_1\colon \mathcal{Z}\to L$ for Example~\ref{ex:n=2}.}
    \label{fig:n=2}
  }
\end{figure}

\clearpage

\begin{proof}[Proof of Proposition~\ref{prop:X(G) homotopy} for $\rank=3$]
\label{proof for rank=3}
The following figure shows all possible combinatorial types of graphs in $\CV[3]$:
\begin{figure}[h]
\begin{center}
\includegraphics[scale=1]{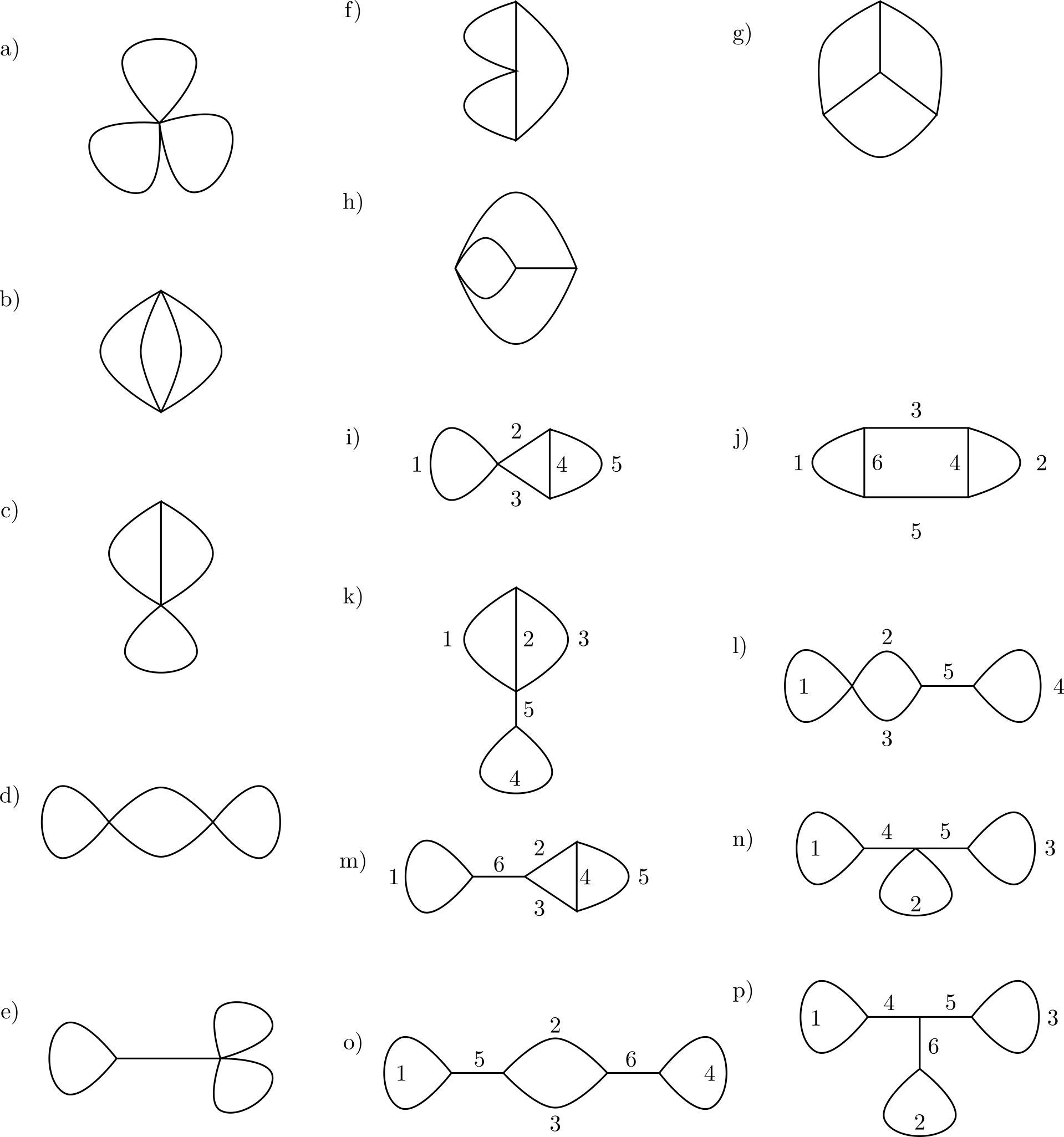}
\end{center}
\caption{All combinatorial types of graphs in $\CV[3]$}
\label{fig:graphs in CV_3}
\end{figure}

We want to show that for each such graph $G$, the poset $\X(G)$ is homotopy equivalent to a (non-trivial) wedge of circles if  $G$ does not contain a separating edge and is contractible otherwise. Using Lemma~\ref{lem:SubgraphstoCores}, it suffices to show the same statement for the poset $\Core(G)$ of all core subgraphs.

If $G$ is a rose, the realisation of $\X(G)=\Sub(G)$ is the boundary of a triangulated $1$-sphere. For the graphs b) -- e) in Figure~\ref{fig:graphs in CV_3}, the complex $\X(G)$ is depicted in Figure~\ref{fig:X(G) 4 edges}.

As the graphs f), g) and h) do not contain any disconnected core subgraphs, the claim here follows from Proposition~\ref{prop:cX(G) homotopy}. The only disconnected core subgraph of i) consists of the edges $1,4$ and $5$. Hence, $\Core(G)$ is derived from $\cCore(G)$ by attaching the star of the vertex $\set[1,4,5]$ along its link. It is an easy exercise to check that the result is homotopy equivalent to a circle. The same is true for j) whose only non-connected core subgraph is $\set[1,2,4,6]$.

For the remaining graphs k) -- p),  the following tables define Morse functions $\phi:\Core(G)\to \mathbb{R}$ with contractible descending links:

\vspace{0.5cm}
k) 
\begin{tabular}{c|c}
vertex $v$ & $\phi(v)$
\\
\hline
$st(\set[1,2,3,4])$& 0 \\
$\cCore(G)\bsl st(\set[1,2,3,4])= \set[\set[1,2,3,5], \,\set[1,2,4,5], \,\set[1,3,4,5]]$& 1
\end{tabular}
\vspace{0.5cm}

l) 
\begin{tabular}{c|c}
vertex $v$ & $\phi(v)$
\\
\hline
$st(\set[1,2,3,4])$& 0 \\
$\cCore(G)\bsl st(\set[1,2,3,4])$& 1
\end{tabular}
\hspace{1cm}
m) 
\begin{tabular}{c|c}
vertex $v$ & $\phi(v)$
\\
\hline
$st(\set[1,2,3,4,5])$& 0 \\
$\cCore(G)\bsl st(\set[1,2,3,4,5])$& 1
\end{tabular}
\vspace{0.5cm}

n) 
\begin{tabular}{c|c}
vertex $v$ & $\phi(v)$
\\
\hline
$st(\set[1,2,3,4])$& 0 \\
$\set[1,2,3,5]$& 1 \\
$\set[2,3,5], \,\set[1,3,4,5]$& 2
\end{tabular}
\hspace{1cm} o) and p)
\begin{tabular}{c|c}
vertex $v$ & $\phi(v)$
\\
\hline
$st(\set[1,2,3,4,5])$& 0 \\
$\set[1,2,3,4,6]$& 1 \\
all other core subgraphs containing $6$ & 2
\end{tabular}

\vspace{0.5cm}
As an illustration, we explain why all the descending links for n) are contractible: $st(\set[1,2,3,4])$ is obviously contractible as this is true for any star in a simplicial complex. The vertices of $\Core(G)$ not contained in this star are precisely the proper core subgraphs of $G$ containing the (separating) edge $5$. The descending link of $\set[1,2,3,5]$ contains a unique maximal element and hence is contractible; this cone point is given by $\set[1,2,3]$ which is the unique maximal core subgraph of $\set[1,2,3,5]$ not containing $5$. As $\set[2,3,5]$ does not contain $4$, it is contained in $\set[1,2,3,5]$ which hence forms a cone point of its descending link. Lastly, the link of $\set[1,3,4,5]$ is coned off by $\set[1,3]$.

The interested reader may complete this argument to an alternative proof of Proposition~\ref{prop:X(G) homotopy} for arbitrary $\rank\geq 2$ in the case where $G$ contains at least one separating edge.
\end{proof} 
\clearpage

\newpage

\begin{figure}[h]
  \centering{ 
  \fontsize{8pt}{8pt}\selectfont
  \def\svgscale{.9}
    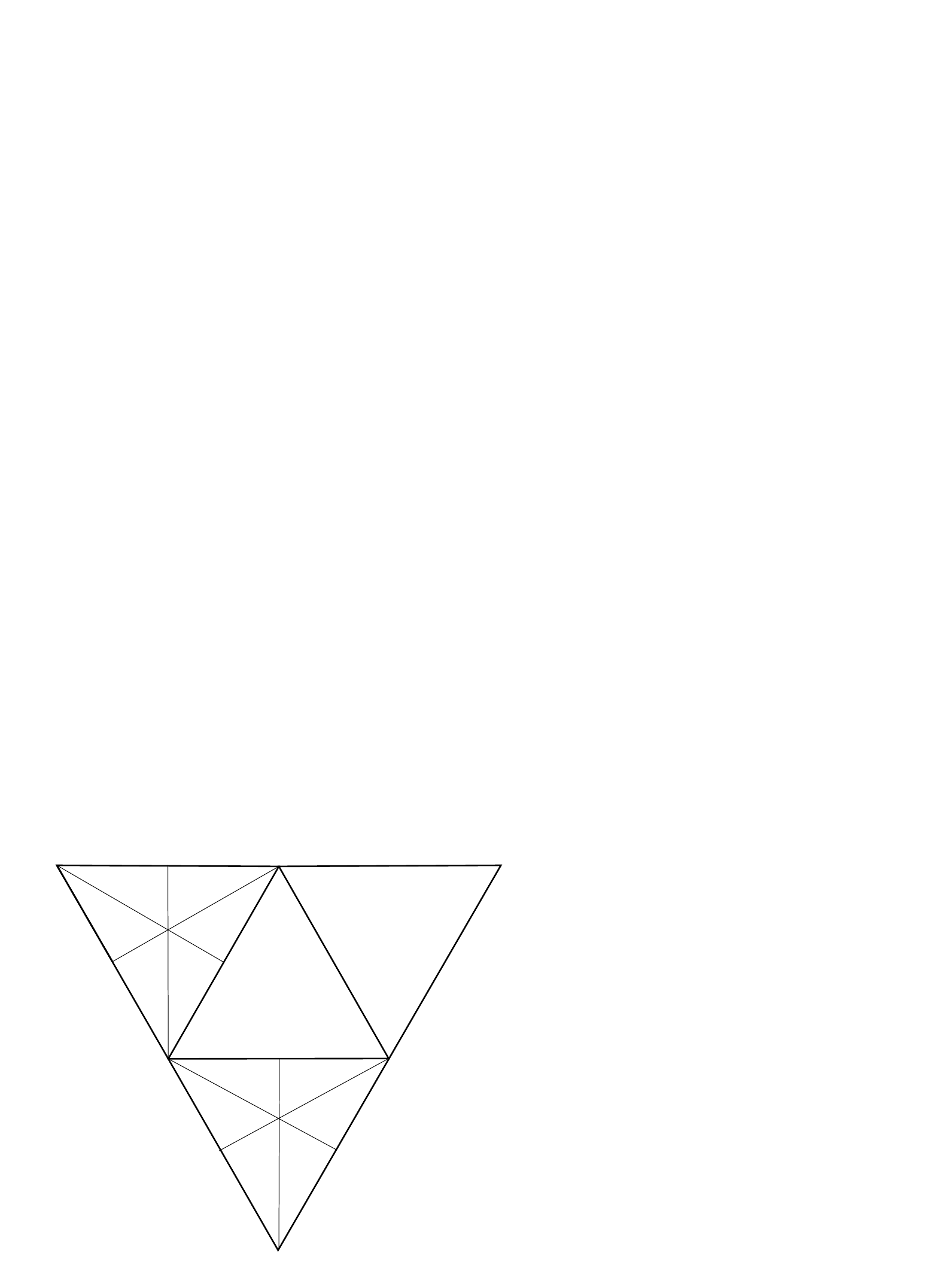
    \caption{The realisation of $\X(G)$ for rank 3 graphs with 4 edges (the tetrahedra were unfolded for better visibility). The first three of them are homotopy equivalent to a wedge of circles while the last one is contractible.}
    \label{fig:X(G) 4 edges}
  }
\end{figure}
\clearpage

\bibliographystyle{alpha}
\bibliography{Bib}

\end{document}